\documentclass[12pt,reqno]{amsart}
\usepackage{amsthm, amsmath, amsfonts, amssymb, graphicx, tikz-cd, makecell, MnSymbol, mathrsfs}%,calrsfs
\usepackage[shortlabels]{enumitem}
\usepackage[hidelinks]{hyperref}
\usepackage{needspace}
\usepackage[capitalize,noabbrev]{cleveref}  
\usepackage{comment}
\usepackage{float}
\usepackage{mathtools}

\newenvironment{proofclaim}{\paragraph{\emph{Proof of the Claim}.}}{\hfill\\} % redefined to use \qedhere

\newtheorem{theorem}{Theorem}[section]
\newtheorem{lemma}[theorem]{Lemma}
\newtheorem{proposition}[theorem]{Proposition}
\newtheorem{corollary}[theorem]{Corollary}

\newtheorem{claim}[theorem]{Claim}
\theoremstyle{definition}
\newtheorem{definition}[theorem]{Definition}
\newtheorem{example}[theorem]{Example}

\newtheorem{remark}[theorem]{Remark}

\newtheorem{convention}[theorem]{Convention}

\crefname{theorem}{Theorem}{Theorems}
\crefname{lemma}{Lemma}{Lemmas}
\crefname{proposition}{Proposition}{Propositions}
\crefname{corollary}{Corollary}{Corollaries}
\crefname{conjecture}{Conjecture}{Conjectures}
\crefname{definition}{Definition}{Definitions}
\crefname{example}{Example}{Examples}
\crefname{question}{Question}{Questions}
\crefname{remark}{Remark}{Remarks}
\crefname{claim}{Claim}{Claims}
\crefname{notation}{Notation}{Notations}
\crefname{convention}{Convention}{Conventions}

\AddToHook{env/lemma/begin}{\crefalias{theorem}{lemma}}
\AddToHook{env/proposition/begin}{\crefalias{theorem}{proposition}}
\AddToHook{env/corollary/begin}{\crefalias{theorem}{corollary}}
\AddToHook{env/conjecture/begin}{\crefalias{theorem}{conjecture}}
\AddToHook{env/definition/begin}{\crefalias{theorem}{definition}}
\AddToHook{env/example/begin}{\crefalias{theorem}{example}}
\AddToHook{env/question/begin}{\crefalias{theorem}{question}}
\AddToHook{env/remark/begin}{\crefalias{theorem}{remark}}
\AddToHook{env/claim/begin}{\crefalias{theorem}{claim}}
\AddToHook{env/notation/begin}{\crefalias{theorem}{notation}}

\newcommand{\thd}{{\twoheaddownarrow}}

\newcommand{\Nach}{{\sf Nach}}
\newcommand{\KHaus}{{\sf KHaus}}
\newcommand{\spbal}{\boldsymbol{\mathit{spba}\ell}}
\newcommand{\uspbal}{\boldsymbol{\mathit{uspba}\ell}}

\newcommand{\bal}{\boldsymbol{\mathit{ba}\ell}}
\newcommand{\ubal}{\boldsymbol{\mathit{uba}\ell}}
\newcommand{\pbal}{\boldsymbol{\mathit{pba}\ell}}
\newcommand{\nbal}{\boldsymbol{\mathit{nba}\ell}}
\newcommand{\unbal}{\boldsymbol{\mathit{unba}\ell}}
\newcommand{\sbal}{\boldsymbol{sbal}}

\newcommand{\usbal}{\boldsymbol{usbal}}

\newcommand{\C}{\mathscr{C}}

\newcommand{\B}{\mathscr{B}}
\newcommand{\R}{\mathscr{R}}
\renewcommand{\P}{\mathscr{P}}

\newcommand{\Q}{\mathscr{Q}}

\newcommand{\X}{\mathscr{X}}
\newcommand{\Y}{\mathscr{Y}}

\makeatletter
\@namedef{subjclassname@2020}{\textup{2020} Mathematics Subject Classification}
\makeatother

\setlist[enumerate,1]{label={\upshape(\arabic*)}}

\makeatletter %defining \plabel
\edef\plabelformat{\string#2\string#1\string#3}
\edef\plabelrangeformat{\string#3\string#1\string#4--\string#5\string#2\string#6}
\newcommand{\plabel}[1]{\label{#1}
\immediate\write\@auxout{\noexpand\crefformat{#1}{\noexpand\cref{#1}\plabelformat}
\noexpand\crefmultiformat{#1}{\noexpand\cref{#1}\plabelformat}{,\plabelformat}{,\plabelformat}{,\plabelformat}
\noexpand\crefrangeformat{#1}{\noexpand\cref{#1}\plabelrangeformat}}}
\makeatother

\begin{document}

\title{Generalizing Gelfand duality to Nachbin spaces}

\author{G.~Bezhanishvili}
\address{New Mexico State University}
\email{guram@nmsu.edu}

\author{P.~J.~Morandi}
\address{New Mexico State University}
\email{pmorandi@nmsu.edu}


\begin{abstract}
We introduce the notion of a Nachbin proximity on a bounded archimedean $\ell$-algebra ($\bal$-algebra). We prove that Gelfand duality lifts to yield a dual equivalence between the categories of uniformly complete $\bal$-algebras equipped with a closed Nachbin proximity and of Nachbin spaces (compact ordered spaces). The key ingredients of the proof include appropriate generalizations of the Stone-Weierstrass theorem and Dieudonn\'{e}'s lemma. We also develop an alternate approach by means of bounded archimedean $\ell$-semialgebras ($\sbal$-algebras), from which we derive De Rudder--Hansoul duality. 
\end{abstract}

\maketitle
\tableofcontents

\section{Introduction}

Compact Hausdorff spaces are one of the central objects of study in topology. The category $\KHaus$ of compact Hausdorff spaces and continuous maps affords numerous dualities, among which Gelfand duality plays a prominent role. It can be developed either by working with the ring of complex-valued functions on compact Hausdorff spaces (commutative $C^*$-algebras) or with the ring of real-valued functions (bounded archimedean $\ell$-algebras). We will work with the latter formalism and refer to \cite{BMO13a} for details about how the two approaches are related to each other. 

Nachbin was one of the pioneers in incorporating order in the study of topology (see \cite{Nac65} and the references therein). In this more general setting, the role of compact Hausdorff spaces is played by compact spaces equipped with a closed order, and it is now common to refer to such as Nachbin spaces. 
It is well known that the category $\Nach$ of Nachbin spaces and continuous order-preserving maps is isomorphic to the category of stably compact spaces \cite[Ch.~VI.6]{GHKLMS03}, and hence to the category of compact regular bitopological spaces. 
Thus, $\Nach$ affords numerous dualities, which are useful in pointfree topology (see, e.g., \cite[Thm.~VI.7.4]{GHKLMS03} as well as \cite{BBH83,Smy92,JS95,BH14}). 

In \cite{DH18}, a version of Gelfand duality was developed for $\Nach$, where the authors worked with  continuous order-preserving functions on Nachbin spaces that are positive.  Axiomatizing the resulting lattice-ordered semirings requires an additional axiom corresponding to the fact that for each such $f$ and a positive real number $r$ with $r \le f$, there is another such $g$ with $f = g + r$. 

We take another approach. Firstly, we find it more natural to work with all continuous order-preserving functions on $X\in\Nach$, not only positive ones. As we will see, an axiomatization of such algebras is more straightforward. In particular, it does not require the additional axiom above.
In the last section, we will see that the resulting category is equivalent to that considered in \cite{DH18}. But we also go one step further and recognize the algebra $C_\le(X)$ of all continuous order-preserving functions on $X$ as the fixpoints of the proximity $\prec_X$ on $C(X)$ given by 
\[
f \prec_X g \iff \textrm{ there is } c \in C_\le(X) \textrm{ with }f \le c \le g.
\]
We provide an axiomatization of the resulting pairs $(A,\prec)$ by introducing the notion of a Nachbin proximity on a $\bal$-algebra. We then lift Gelfand duality by establishing that $\Nach$ is dually equivalent to the category of such pairs, where $A$ is uniformly complete and $\prec$ is a Nachbin proximity on $A$ that is closed in $A \times A$. This is done by obtaining appropriate generalizations of the Stone-Weierstrass theorem and Dieudonn\'{e}'s lemma to the setting of Nachbin spaces. 

As a result, we obtain two generalizations of Gelfand duality to $\Nach$, one using the formalism of Nachbin proximities on $\bal$-algebras and another using 
that of bounded archimedean $\ell$-semialgebras. The latter is closely related to the approach of \cite{DH18}, which we obtain as a consequence of our approach.

The paper is organized as follows. In \cref{sec: preliminaries}, we briefly recall Gelfand duality for compact Hausdorff spaces using the formalism of $\bal$-algebras (bounded archimedean $\ell$-algebras).
In \cref{sec: sbal}, we generalize the notion of a $\bal$-algebra to that of an $\sbal$-algebra and show that $\bal$ is a reflective subcategory of $\sbal$. 
In \cref{sec: pbal}, we introduce the notion of proximity on a $\bal$-algebra, giving rise to the category $\pbal$ of proximity $\bal$-algebras. We construct an adjoint pair of functors between $\pbal$ and $\sbal$, yielding an equivalence between $\sbal$ and a full subcategory $\spbal$ of $\pbal$ consisting of skeletal proximity $\bal$-algebras. 
In \cref{sec: Nachbin proximities}, we further refine the notion of proximity to that of Nachbin proximity and show that there is a contravariant adjunction between the category $\nbal$ of $\bal$-algebras equipped with a Nachbin proximity and the category $\Nach$ of Nachbin spaces.   
In \cref{sec: duality}, we introduce the full subcategory $\unbal$ of $\nbal$ consisting of uniformly complete $\bal$-algebras equipped with a closed Nachbin proximity, as well as the full subcategory $\usbal$ of $\sbal$ consisting of uniformly complete $\sbal$-algebras. Our main result establishes that $\unbal$ and $\usbal$ are equivalent to each other and dually equivalent to $\Nach$. It is obtained by developing appropriate generalizations of the Stone-Weierstrass theorem and Dieudonn\'{e}'s lemma for Nachbin spaces. The section ends with a diagram detailing the equivalences and dual equivalences obtained in the paper.  
Finally, in \cref{sec: comparison with DH18}, we compare our approach to that of \cite{DH18}.

\section{Preliminaries} \label{sec: preliminaries}

{\em An ordered space} is a topological space equipped with a partial order. A {\em Nachbin space} is a compact ordered space $X$ such that the order $\le$ is closed in $X \times X$. The latter condition is equivalent to the following separation axiom (see \cite[Prop.~1]{Nac65}): 
\begin{align*} \label{eqn: COR}
\mbox{If $x \not\le y$, then there exist disjoint neighborhoods $U$ of $x$} \\
\mbox{and $V$ of $y$ such that $U$ is an upset and $V$ is a downset.} 
\end{align*}

This separation axiom is known as \emph{order-Hausdorffness} (see \cite{McC68}).

\begin{definition}
Let $\Nach$ be the category of Nachbin spaces and continuous order-preserving maps.    
\end{definition}

The category $\Nach$ serves as an order-topological analog of $\KHaus$. As we pointed out in the introduction, our aim is to lift Gelfand duality to $\Nach$.  We start by recalling some basic definitions (see, e.g., \cite[Ch.~XIII--XVII]{Bir79}).

All rings are assumed to be commutative with $1$ and ring homomorphisms are assumed to be unital (preserve $1$).  An {\em $\ell$-ring} (lattice-ordered ring) is a ring $A$ with a lattice order $\le$ such that $a\le b$ implies $a+c \le b+c$ and $0 \leq a, b$ implies $0 \le ab$. An $\ell$-ring $A$ is \emph{bounded} if for each $a \in A$ there is $n \in \mathbb{N}$ with $-n\cdot 1 \le a \le n\cdot 1$ ($1$ is a strong order unit), and it is \emph{archimedean} if whenever $n\cdot a \le b$ for each $n \in \mathbb{N}$, then $a \le 0$.

An $\ell$-ring $A$ is an \emph{$\ell$-algebra} if it is an $\mathbb R$-algebra and for each $0 \le a\in A$ and $0\le r\in\mathbb R$ we
have $0 \le r \cdot a$. Let $\bal$ be the category of bounded archimedean $\ell$-algebras and unital $\ell$-algebra homomorphisms.

For each $A \in \bal$, we view $\mathbb{R}$ as an $\ell$-subalgebra of $A$ by identifying
$r\in\mathbb R$ with $r\cdot 1\in A$. For $a \in A$, define the {\em positive} and {\em negative} parts of $a$ by 
\[
a^+ = a \vee 0 \mbox{ and } a^- = (-a) \vee 0.
\]

\begin{remark}
Our definition of $a^-$ follows that of \cite[p.~75]{HJ61}. A more standard definition is $a \wedge 0$ (see \cite[p.~293]{Bir79}). The two are related by $a^- = -(a\wedge 0)$. 
\end{remark}

We have $a^+, a^- \ge 0$ and $a = a^+ - a^-$ (see, e.g., \cite[p.~75]{HJ61}).
Next, define the \emph{absolute value} and {\em norm} of $a$ by
\begin{equation} \label{eqn: norm}
|a|=a\vee(-a) \quad \mbox{and} \quad  
||a||=\inf\{r\in\mathbb{R} : |a|\le r\}.
\end{equation}
Then $A$ is \emph{uniformly complete} if the norm is complete. Let $\ubal$ be the full subcategory of $\bal$ consisting of uniformly complete $\ell$-algebras. The following result is known as Gelfand duality as it originates in \cite{Gel39} (see also \cite{Sto40,GN43}). 

\begin{theorem} [Gelfand duality]
There is a contravariant adjunction between $\bal$ and $\KHaus$ which restricts to a dual equivalence between $\KHaus$ and $\ubal$.
\[
\begin{tikzcd}
\ubal \arrow[rr, hookrightarrow] && \bal \arrow[dl, "\X"]  \arrow[ll, bend right = 20, "\C\X"'] \\
&  \KHaus \arrow[ul,  "\C"] &
\end{tikzcd}
\]
\end{theorem} 

We briefly describe the functors establishing Gelfand duality. Our account follows that of \cite{HJ61,BMO13a}. 
The functor $\C:\KHaus \to \bal$ associates with each compact Hausdorff space $X$ the ring $C(X)$ of (necessarily bounded) continuous real-valued functions on $X$, and with each continuous map $\varphi:X\to Y$ the unital $\ell$-algebra homomorphism $\C \varphi : C(Y)\to C(X)$ given by $\C \varphi (f)=f\circ\varphi$. 

To define the functor $\X:\bal\to\KHaus$, we recall that an ideal $I$ of $A\in\bal$ is an \emph{$\ell$-ideal} if $|a|\le|b|$ and $b\in I$ imply $a\in I$. 
Let $X_A$ be the space of maximal $\ell$-ideals of $A$, whose closed sets are exactly sets of the form
\[
Z_\ell(I) = \{x\in X_A: I\subseteq x\},
\]
where $I$ is an $\ell$-ideal of $A$.
The functor $\X : \bal \to \KHaus$ 
associates with each $A \in \bal$ the space $X_A$ and with each morphism $\alpha : A \to B$ in $\bal$ the continuous map $\X\alpha = {\alpha^{-1} : X_B \to X_A}$. 

We next recall the units  $\eta:1_{\KHaus}\to\X\circ\C$ and $\phi:1_{\bal}\to\C\circ\X$. 
For $X\in\KHaus$ and $x\in X$ let $M_x:=\{f\in C(X) : f(x)=0\}$ be the maximal $\ell$-ideal of $C(X)$ consisting of functions that vanish at $x\in X$, and define 
$\eta_X:X\to X_{C(X)}$ by 
$
\eta_X(x)=M_x.
$
Then $\eta_X$ is a homeomorphism, so $\eta:1_{\KHaus}\to\X\circ\C$ is a natural isomorphism. 

Let $A\in\bal$. For $x\in X_A$, the quotient $A/x$ is 
isomorphic to $\mathbb R$.
Therefore, for each $a\in A$ there is a unique $r\in\mathbb R$ with $a+x=r+x$. 
Thus, we may define $\phi_A :A\to C(X_A)$ by 
$
\phi_A(a)(x)=r
$
for each $x\in X_A$. 
Consequently, $\phi:1_{\bal}\to\C\circ\X$ is a natural transformation. 

Moreover, for $A \in \bal$ and $X \in \KHaus$, we have a natural bijection
\[
\theta : \hom_{\bal}(A, \C X) \to \hom_{\KHaus}(X, \X A)
\]
given by $\theta(\alpha) = \chi_\alpha \circ \eta_X$ for each $\alpha \in \hom_{\bal}(A, \C X)$, yielding that $\C$ and $\X$ form a contravariant adjunction.

Furthermore, since $\bigcap X_A = 0$, 
$\phi_A$ is a monomorphism in $\bal$. In addition, $\phi_A$ separates points of $X_A$, meaning that for each $x,y\in X_A$ with $x \ne y$, there is $a \in A$ such that $\phi_A(a)(x) \ne \phi_A(a)(y)$. Therefore, 
the Stone-Weierstrass theorem yields that if $A$ is uniformly complete, then $\phi_A$ is an isomorphism. In what follows, we will repeatedly use the following result  (see, e.g., \cite[Prop.~3.3]{BMO13a}). 

\begin{proposition} \label{prop: equivalent conditions for density} Let $A,B \in \bal$ and $\alpha : A \to B$ be a 
monomorphism in $\bal$. The following are equivalent.
\begin{enumerate}
    \item $\alpha[A]$ is uniformly dense in $B$.
    \item $\alpha[A]$ separates points of $X_B$.
    \item $\X \alpha : \X B \to \X A$ is a homeomorphism. 
\end{enumerate} 
\end{proposition}

\section{\texorpdfstring{$\sbal$}{sbal}-algebras} \label{sec: sbal}

As we saw in the previous section, there is a natural ring to associate with each compact Hausdorff space $X$, namely the ring $C(X)$ of all continuous real-valued functions on $X$. If $(X, \le)$ is a Nachbin space, then it is also natural to consider $$C_\le(X) = \{ f \in C(X) : f \mbox{ is order preserving} \}.$$ 
In order to axiomatize $C_\le(X)$, we recall:

\begin{definition}\plabel{def: ell monoid}
\cite[Def XIV.4.4]{Bir79}  Let $(S, +, 0)$ be a commutative monoid with a lattice order $\le$. We call $S$ an \emph{$\ell$-monoid} if for all $a, b, c \in S$,
\begin{enumerate}[label=(S\arabic*)]
\item \label[def: ell monoid]{S1}$a \le b \iff a + c \le b + c$;
\item \label[def: ell monoid]{S2} $(a \vee b) + c = (a + c) \vee (b + c)$;
\item \label[def: ell monoid]{S3} $(a \wedge b) + c = (a + c) \wedge (b + c)$.
\end{enumerate}
\end{definition}

\begin{remark} \label{rem: sbals are cancellative}
\ref{S1} shows that an $\ell$-monoid is a cancellative monoid. This will be used often in what follows.
\end{remark}

\begin{definition} \label{def: ell monoid morphism}
A map $\alpha : S \to T$ between $\ell$-monoids is an $\ell$-\emph{monoid morphism} provided that $\alpha$ is a lattice morphism which preserves addition. 
\end{definition}

In particular, since $T$ is a cancellative monoid, $\alpha(0) = 0$ for each $\ell$-monoid morphism $\alpha : S \to T$. For an $\ell$-monoid $S$, let $S^+ = \{ a \in S : a \ge 0\}$ be the {\em positive cone} of $S$. The next definition is a modification of \cite[Def.~2.1]{DH18}.

\begin{definition} \plabel{def: ell semialgebra}
An $\ell$-monoid $S$ is an $\ell$-\emph{semialgebra} if 
\begin{enumerate}[label=(S\arabic*)]
\setcounter{enumi}{3}
\item \label[def: ell semialgebra]{S4} there is a commutative, associative multiplication on $S^+$ which has multiplicative identity 1; 
\item \label[def: ell semialgebra]{S5} $a(b+c) = ab + ac$ for all $a,b,c \in S^+$;
\item \label[def: ell semialgebra]{S6} $a \le b$ implies $ac \le bc$ for all $a, b, c \in S^+$;
\item \label[def: ell semialgebra]{S7} there is a unital $\ell$-monoid morphism $\rho : \mathbb{R} \to S$ preserving multiplication.
\end{enumerate}
\end{definition}

\begin{remark}
If $S$ is a nonzero $\ell$-semialgebra with $\rho : \mathbb{R} \to S$ the associated morphism, then 
$\rho$ is one-one. To see this,
let $r,s \in \mathbb{R}$ with $\rho(r) = \rho(s)$. Without loss of generality, we may assume that $r \ge s$. Then $r = s + t$ with $t \ge 0$. Therefore, $\rho(r) = \rho(s) + \rho(t)$, and so $\rho(t) = 0$. If $t > 0$, then there is $n \in \mathbb{N}$ with $nt \ge 1$. Since $\rho$ is order preserving, $\rho(nt) \ge \rho(1) = 1$. On the other hand, $\rho(nt) = n\rho(t) = 0$. The obtained contradiction proves that $t = 0$, and hence $r = s$.
\end{remark}

In view of the above remark, we make the following:

\begin{convention} \label{conv: R inside S}
Since the results in this article hold easily for the zero monoid, in the proofs we will skip the easy verification for the zero monoid. 
For a nonzero $\ell$-semialgebra $S$ we identify $\mathbb{R}$ with $\rho[\mathbb{R}]$ and view $\mathbb{R}$ inside $S$. For $a \in S$ and $r \in \mathbb R^+$, we write $a - r$ for $a + (-r)$. 
\end{convention}

\begin{definition} \plabel{def: semibal}
An $\ell$-semialgebra $S$ is 
\begin{enumerate}[label=(S\arabic*)]
\setcounter{enumi}{7}
\item \emph{bounded} if for each $a \in S$ there is $r \in \mathbb{R}^+$ with $-r\le a \le r$;\label[def: semibal]{S8}
\item \emph{archimedean} if for each $a, b, c, d \in S$, if $na + b \le nc + d$ for each $n \in \mathbb{N}$, then $a \le c$.\label[def: semibal]{S9}
\end{enumerate}
\end{definition}

\begin{remark}
The above definitions trivially hold for the zero monoid. The definition of bounded is slightly different from that given in \cite[Def.~2.1]{DH18}, where the authors assume that $S=S^+$. On the other hand, our definition of archimedean is the same as that in \cite[Def.~2.1]{DH18}. 
\end{remark}

We can define an $\mathbb{R}^+$-action on $S^+$ by $r a = \rho(r)a$ for each $r \in \mathbb{R}^+$ and $a \in S^+$. The following properties are easy consequences of the definition, where $a,b\in S^+$ and $r,s\in \mathbb R^+$.
\begin{enumerate}
\item $a \le b$ implies $ra \le rb$;
\item $r(a+b) = ra+rb$ and $(r+s)a = ra + sa$;
\item $r(ab) = (ra)b$;
\item $1_\mathbb{R}\, a = a$.
\end{enumerate}

We can extend the $\mathbb{R}^+$-action to all of $S$ as follows. Let $a \in S$ and $r \in \mathbb{R}^+$. By \ref{S8}, there is $s \in \mathbb{R}^+$ with $a + s \ge 0$, and we set
$$r a = r(a + s) - rs.$$ 
To see that this is well defined, let $a+s, a+t \ge 0$. We have
\[
r(a+s+t) = r(a+s) + rt \quad\textrm{and}\quad r(a+s+t) = r(a+t) + rs. 
\]
Adding $-rs - rt$ to both sides yields $r(a+s) - rs = r(a+t) - rt$.

\begin{definition} \label{def: sbal}
Let $\sbal$ be the category whose objects are bounded archimedean $\ell$-semialgebras and whose morphisms are unital $\ell$-monoid morphisms $\alpha: S\to T$ that preserve multiplication on $S^+$ and the $\mathbb R^+$-action.
\end{definition}

\begin{remark} \label{rem: alpha r = r}
Let $\alpha$ be an $\sbal$-morphism. Since $\alpha$ preserves the $\mathbb R^+$-action, $\alpha(r)=r$ for each $r\in\mathbb R^+$. Moreover, $$0 = \alpha(0) = \alpha(r-r) = \alpha(r) + \alpha(-r) = r + \alpha(-r),$$ so $\alpha(-r) = -r$, and hence $\alpha(r)=r$ for all $r \in \mathbb R$.
\end{remark}

As we pointed out at the beginning of this section, our motivating examples of $\sbal$-algebras are $C_\le(X)$, where $X$ is a Nachbin space. Observe that
\[
\{ f-g : f,g \in C_\le(X)\}
\]
is the $\bal$-subalgebra of $C(X)$ generated by $C_\le(X)$. We think of this algebra as the $\bal$-envelope of $C_\le(X)$ and show that each $\sbal$-algebra has such an envelope. 
The following definition is 
along the lines of \cite[Def.~2.5]{DH18}.

\begin{definition} \label{def: bal envelope}
For $S \in \sbal$, define the {\em $\bal$-envelope} $\B S$ of $S$ to be the quotient of $S \times S$ modulo the equivalence relation $\sim$ given by
\[
(a,b) \sim (c,d) \iff a+d = b+c.
\]
\end{definition}

We denote the equivalence class of $(a,b)$ by $[a,b]$. We show that each element of $\B S$ can be represented in the form $[a,b]$ with $a,b \ge 0$. Let $a, b \in S$. Since $S$ is bounded, there is $r \in \mathbb{R}^+$ with $a+r, b+r \ge 0$. Therefore, $[a,b] = [a+r, b+r]$.   
This allows us to only consider $[a,b],[c,d] \in \B S$ with $a,b,c,d \ge 0$ and define addition and multiplication on $\B S$ by
\begin{align*}
[a,b] + [c,d] &= [a+c, b+d];\\
[a,b] \cdot [c,d] &= [ac + bd, ad + bc].
\end{align*}

We next define multiplication by positive scalars by 
\[
r[a,b] = [ra, rb] \textrm{ for each }r \in \mathbb{R}^+
\]
and extend it to all scalars by
$$-r[a,b] = r[b, a] \textrm{ for each }r \in \mathbb{R}^+.$$

It is straightforward to see that these operations are well defined, and that $\B S$ is an $\mathbb R$-algebra, where the additive and multiplicative identities 
are $[0,0]$ and $[1,0]$, respectively. 

Finally, we define $\le$ on $\B S$ by
\[
[a,b] \le [c,d] \iff a + d \le b + c.
\]
It follows from \ref{S1} that $\le$ is a partial order on $\B S$.

\begin{proposition} \plabel{prop: UMP for envelope}
Let $S \in \sbal$.
\begin{enumerate}
\item \label[prop: UMP for envelope]{BS in bal} $\B S \in \bal$.
\item \label[prop: UMP for envelope]{varepsilon} Define $\varepsilon : S \to \B S$ by $\varepsilon(a) = [a,0]$. Then $\varepsilon$ is an $\sbal$-embedding, and $$\B S = \{ \varepsilon(a) - \varepsilon(b) : a, b \in S\}.$$
\end{enumerate}
\end{proposition} 

\begin{proof}
\ref{BS in bal}
To see that $\B S$ is a 
a lattice, let $a,b,c,d \in S^+$. We claim that
\[
[a,b] \vee [c,d] = [(a+d)\vee(b+c), b + d].
\]
Since $a + d \le (a+d)\vee (b+c)$, we have $a + b + d \le b + \left((a+d)\vee (b+c)\right)$ by \ref{S1}. This yields $[a,b] \le [(a+d)\vee(b+c), b + d]$. Similarly, $[c,d] \le [(a+d)\vee(b+c), b + d]$. Next, let $[e,f] \in \B S$ with $[a,b], [c,d] \le [e,f]$. Then $a+f \le b+e$ and $c+f \le d+e$. Therefore, $a+d+f, c+b+f \le b+d+e$ by \ref{S1}, and so $(a+d+f)\vee (c+b+f) \le b+d+e$. Thus, $\big((a+d)\vee (c+b)\big) + f \le b + d + e$ by \ref{S2}, which yields $[(a+d)\vee(b+c), b + d] \le [e,f]$. This verifies the claim. A similar calculation shows that the meet is given by 
\[
[a,b] \wedge [c,d] = [(a+d)\wedge(b+c), b + d].
\]
Thus, $\B S$ is a lattice, and it is straightforward to see that it is an $\ell$-algebra.

To see that $\B S$ is bounded, let $a, b \in S^+$. Since $S$ is bounded, there are $r, s \in \mathbb{R}^+$ with $-r \le a \le r$ and $-s \le b \le s$. Therefore, $a \le r \le b + (r+s)$ and $b \le s \le a + (r+s)$, and hence 
$-(r+s)[1,0] \le [a,b] \le (r+s)[1,0]$.
Finally, to see that $\B S$ is 
archimedean, let $n[a,b] \le [c,d]$ for each $n \in \mathbb{N}$. Then 
$[na, nb] \le [c, d]$, so $na + d \le nb + c$ for each $n \in \mathbb{N}$. Since $S$ is archimedean, $a \le b$, and hence $[a,b] \le [0,0]$. Thus, $\B S \in \bal$.

\ref{varepsilon} 
If $\varepsilon(a) = \varepsilon(b)$, then $[a,0] = [b,0]$, so $(a,0) \sim (b,0)$, and hence $a = b$. That  $\varepsilon$ is an $\sbal$-morphism is obvious, so $\varepsilon$ is an $\sbal$-embedding. Finally, if $[a,b] \in \B S$, then
\[
[a,b]= [a,0] + [0,b] = [a,0]-[b,0] = \varepsilon(a) - \varepsilon(b). 
\]
Thus, $\B S = \{ \varepsilon(a) - \varepsilon(b) : a, b \in S\}$. 
\end{proof}

\begin{remark}
Let $S\in\sbal$. Because $S$ embeds into its $\bal$-envelope $\B S$, we may define the norm on $S$ by restricting the norm $\|\cdot\|$ on $\B S$ to $S$.
\end{remark}

\begin{theorem} \label{thm: bal is reflective}
$\bal$ is a full reflective subcategory of $\sbal$.
\end{theorem}

\begin{proof}
We first observe that $\bal$ is a full subcategory of $\sbal$. Since each $\bal$-algebra is a $\sbal$-algebra, it is enough to show that each $\sbal$-morphism $\alpha : A \to B$, with $A, B \in \bal$, is a $\bal$-morphism. For this, it is sufficient to show that $\alpha$ preserves multiplication. Let $a, b \in A$. Then $a = a^+ - a^-$ and $b = b^+ - b^-$. Therefore, $\alpha(a) + \alpha(a^-) = \alpha(a + a^-) = \alpha(a^+)$, so $\alpha(a) = \alpha(a^+) - \alpha(a^-)$. Similarly, $\alpha(b) = \alpha(b^+) - \alpha(b^-)$.
This together with $\alpha$ preserving multiplication of positive elements yields
\begin{align*}
\alpha(ab) &= \alpha[(a^+ - a^-)(b^+ - b^-)] = \alpha[a^+b^+ + a^-b^- - (a^+b^- + a^-b^+)] \\
&= \alpha(a^+b^+) + \alpha(a^-b^-) - [\alpha(a^+b^-) + \alpha(a^-b^+)] \\
&= \alpha(a^+)\alpha(b^+) + \alpha(a^-)\alpha(b^-) - [\alpha(a^+)\alpha(b^-) + \alpha(a^-)\alpha(b^+)] \\
&= [\alpha(a^+) - \alpha(a^-)][\alpha(b^+) - \alpha(b^-)] \\
&= \alpha(a)\alpha(b).
\end{align*}
This proves that $\bal$ is a full subcategory of $\sbal$.
To see that $\bal$ is a reflective subcategory of $\sbal$, 
let $\alpha : S \to A$ be an $\sbal$-morphism with $A \in \bal$. Define $\beta : \B S \to A$ by $$\beta[a,b] = \alpha(a) - \alpha(b)$$ for each $a,b \in S^+$. 
It is straightforward to see that $\beta$ is a well-defined $\bal$-morphism. The diagram below commutes since $\beta\varepsilon(a) = \beta[a,0] = \alpha(a) - \alpha(0) = \alpha(a)$, and $\beta$ is unique by \cref{varepsilon}.
\[
\begin{tikzcd}[column sep = 5pc]
 S \arrow[r, "\varepsilon"] \arrow[dr, "\alpha"'] & \B S \arrow[d, "\beta"] \\
& A
\end{tikzcd}
\]
Thus, $\bal$ is a reflective subcategory of $\sbal$ by \cite[p.~89]{Mac71}.
\end{proof}

\begin{remark} \label{rem: B is a functor}
As follows from the above, taking the $\bal$-envelope of an $\sbal$-algebra extends to a functor $\B:\sbal\to\bal$ which is left adjoint to the inclusion functor $\bal\hookrightarrow\sbal$. The action of $\B$ on an $\sbal$-morphism $\alpha : S \to T$ is given by 
$\B \alpha[a,b] = [\alpha(a), \alpha(b)]$ for each $a,b \in S$.
\end{remark}

\begin{corollary} \label{cor: UMP for envelope}
Let $S \in \sbal$, $A \in \bal$, and $\alpha : S \to A$ be a one-to-one $\sbal$-morphism. Then ${\{ \alpha(a) - \alpha(b) : a, b \in S\}}$ is the $\bal$-subalgebra of $A$ generated by $\alpha[S]$ which is isomorphic to~$\B S$.
\end{corollary}

\begin{proof}
By \cref{thm: bal is reflective}, there is a $\bal$-morphism $\beta : \B S \to A$ satisfying $\beta \circ \varepsilon = \alpha$. To see that $\beta$ is one-one, let $\varepsilon(a) - \varepsilon(b) \in \ker(\beta)$. Then $\alpha(a) - \alpha(b) = 0$, so $\alpha(a) = \alpha(b)$, and hence $a=b$ because $\alpha$ is one-to-one. Therefore, $\varepsilon(a) - \varepsilon(b) = 0$. Thus, $\beta$ is a $\bal$-isomorphism between $\B S$ and $\{ \beta[a,b] : a,b \in S\}$. Since the latter is equal to $\{ \alpha(a) - \alpha(b) : a, b \in S\}$, we conclude that $\{ \alpha(a) - \alpha(b) : a, b \in S\}$ is the $\bal$-subalgebra of $A$ generated by $\alpha[S]$.
\end{proof} 

\begin{convention} \label{con: BS inside A}
Because of \cref{cor: UMP for envelope}, if $S\in\sbal$ is an $\sbal$-subalgebra of $A\in\bal$, we identify $\B S$ with its image in $A$. 
\end{convention}

\section{Proximities on \texorpdfstring{$\bal$}{bal}-algebras}\label{sec: pbal}

As we pointed out in the previous section, a typical example of an $\sbal$-algebra is the algebra $C_\le(X)$ of continuous order-preserving functions on a Nachbin space $X$. Such functions can be recognized as reflexive elements of the following binary relation on $C(X)$: 
\begin{equation}
f \prec_X g \iff \textrm{there is } c \in C_\le(X) \textrm{ with }f \le c \le g. \label{eqn: le sub X}
\end{equation}
This motivates the following definition. 

\begin{definition}\ \label{def: proximity on bal}
\begin{itemize}[leftmargin=.4in]
\item {\bf Proximity on a lattice.} Let $A$ be a lattice. A binary relation $\prec$ on $A$ is a {\em proximity} if it satisfies: 
\end{itemize}
\begin{enumerate}[label=(P\arabic*)]
\item \label{P1}If $a \prec b$, then $a \le b$.
\item \label{P2} If $a \le b \prec c \le d$, then $a \prec d$.
\item \label{P3} If $a \prec b, c$, then $a \prec b \wedge c$.
\item \label{P4} If $a, b \prec c$, then $a \vee b \prec c$.\item \label{P5} If $a \prec b$, then there is $c$ with $a \prec c \prec b$.
\end{enumerate}
\begin{itemize}[leftmargin=.35in]
\item {\bf Proximity on an $\ell$-ring.} Let $A$ be an $\ell$-ring. Then $\prec$ is a {\em proximity on $A$} if in addition it satisfies 
\end{itemize}
\begin{enumerate}[label=(P\arabic*)]
\setcounter{enumi}{5}
\item \label{P6} If $a \prec b$ and $c \prec d$, then $a + c \prec b +d$.
\item \label{P7} If $0 \le a, b, c, d$ with $a \prec b$ and $c \prec d$, then $ac \prec bd$.
\end{enumerate}
\begin{itemize}[leftmargin=.4in]
\item {\bf Proximity on an $\ell$-algebra.} Let $A$ be an $\ell$-algebra. Then $\prec$ is a {\em proximity on $A$} if it further satisfies 
\end{itemize}
\begin{enumerate}[label=(P\arabic*)]
\setcounter{enumi}{7}
\item \label{P8} If $r \in \mathbb{R}$, then $r \prec r$.
\item \label{P9} If $a \prec c$ and $r \in \mathbb{R}^+$, then $r a \prec r b$.
\end{enumerate}
\end{definition}

We point out that \ref{P1}--\ref{P5} are standard in topology 
(see, e.g., \cite[Thm.~3.9]{NW70}), \ref{P6}--\ref{P7} govern the interaction of $\prec$ with ring operations, and \ref{P8}--\ref{P9} 
that with $\mathbb{R}$-action (see, e.g., \cite{BMMO15b}).
We will mainly work with proximities on $\bal$-algebras. In addition to the above nine basic axioms, we will require the following. 

\begin{definition}
We call a proximity $\prec$ on $A \in \bal$ \emph{reflexive} if it satisfies the following strengthening of \ref{P5}:
\begin{enumerate}[label=(RP\arabic*)]
\setcounter{enumi}{4}
\item \label{RP5} If $a \prec b$, then there is $c$ with $a \prec c \prec c \prec b$.
\end{enumerate}
\end{definition}

\begin{remark} \label{rem: proximity diagram}
It is worth pointing out that if $a$ is reflexive, then $a^+$ is reflexive, but $a^-$ need not be reflexive. 
\end{remark}

\begin{definition}
A {\em proximity $\bal$-algebra} is a pair $(A,\prec)$, where $A \in \bal$ and $\prec$ is a reflexive proximity on $A$.  
Let $\pbal$ be the category whose objects are proximity $\bal$-algebras 
and whose morphisms are $\bal$-morphisms $\alpha: A\to B$ that preserve  proximity; that is, $$a \prec b \Longrightarrow \alpha(a) \prec \alpha(b) \mbox{ for all } a,b \in A.$$
\end{definition}

\begin{remark}
In \cref{rem: comparison}, we compare proximity $\bal$-algebras to other related notions in the literature. 
\end{remark}

We define a functor from $\pbal$ to $\sbal$.

\begin{definition}
For $(A,\prec) \in \pbal$ let $\R(A, \prec) = \{ c \in A : c \prec c\}$, and for a $\pbal$-morphism $\alpha: A\to B$ let $\R\alpha$ be the restriction of $\alpha$ to $\R(A,\prec)$.
\end{definition}

\begin{lemma} \plabel{lem: properties of nbal morphisms}
Let $(A, \prec), (B, \prec) \in \pbal$ and $\alpha : A \to B$ be a $\bal$-morphism.
\begin{enumerate}
\item \label[lem: properties of nbal morphisms]{bal morphism being nbal} $\alpha$ is a $\pbal$-morphism iff $\alpha[\R(A,\prec)] \subseteq \R(B, \prec)$.
\item \label[lem: properties of nbal morphisms]{nbal iso} $\alpha$ is a $\pbal$-isomorphism iff it is a $\bal$-isomorphism which preserves and reflects proximity.
\end{enumerate}
\end{lemma}

\begin{proof}
\ref{bal morphism being nbal} Suppose $\alpha$ is a $\pbal$-morphism and $c \in \R(A, \prec)$. Then $c \prec c$, so $\alpha(c) \prec \alpha(c)$. Therefore, $\alpha(c) \in \R(B, \prec)$. Conversely, suppose $\alpha[\R(A, \prec)] \subseteq \R(B, \prec)$. Let $a, b \in A$ with $a \prec b$. Then there is $c \in \R(A, \prec)$ with $a \le c \le b$. Therefore, $\alpha(a) \le \alpha(c) \le \alpha(b)$. Since $\alpha(c) \in \R(B, \prec)$, we conclude that $\alpha(a) \prec \alpha(b)$. Thus, $\alpha$ is a $\pbal$-morphism.

\ref{nbal iso} Recalling that $\bal$-isomorphisms are bijective $\bal$-morphisms, 
$\alpha$ is a $\bal$-isomorphism preserving and reflecting proximity iff $\alpha$ and $\alpha^{-1}$ are both $\pbal$-morphisms, yielding the result. 
\end{proof}

\begin{proposition} \label{prop: functor R}
$\R:\pbal\to\sbal$ is a well-defined functor. 
\end{proposition}

\begin{proof}
 By \ref{P3} and \ref{P4}, $\R(A,\prec)$ is closed under join and meet; by \ref{P6}, $\R(A,\prec)$ is closed under addition; and by \ref{P7}, $\R(A,\prec)^+$ is closed under multiplication. Finally, $\R(A,\prec)$ is closed under positive scalar multiplication by \ref{P9}. Therefore, 
 $\R(A,\prec) \in \sbal$, and so
 $\R$ is well defined on objects. By \cref{bal morphism being nbal}, $\R$ is also well defined on morphisms, and it clearly preserves composition and identity morphisms. Thus, $\R$ is a well-defined functor.
\end{proof}

To define a functor in the other direction, we require the following:

\begin{definition} \label{ex: proximity from a sub}
Let $A \in \bal$ and $S$ be an $\sbal$-subalgebra of $A$.
Define $\prec_S$ on $A$ by
\[
a \prec_S b \iff \exists s \in S : a \le s \le b. 
\] 
\end{definition}

The proof of the following lemma is straightforward and we skip it.

\begin{lemma} \label{lem: ref prox from semibal}
Let $A \in \bal$ and $S$ be an $\sbal$-subalgebra of $A$. Then $\prec_S$ is a reflexive proximity on $A$ such that $\R(A, \prec_S) = S$.
\end{lemma}

\begin{proposition} \label{prop: functor B}
There is a functor $\B : \sbal \to \pbal$.
\end{proposition}

\begin{proof}
For simplicity, we identify each $S \in \sbal$ with its image $\varepsilon[S] \subseteq \B S$. Now associate with each $S$ the pair $(\B S, \prec_S)$ and with each $\sbal$-morphism $\alpha: S \to T$ the morphism $\B\alpha : \B S \to \B T$. By \cref{lem: ref prox from semibal}, $(\B S, \prec_S) \in \pbal$. By \cref{rem: B is a functor}, $\B \alpha$ is a $\bal$-morphism, and $\B\alpha$ 
preserves $\prec_S$ by \cref{bal morphism being nbal}. Finally, it is clear that $\B$ preserves 
composition and identity morphisms.
\end{proof}

Let $(A, \prec) \in \pbal$. By \cref{con: BS inside A}, we identify $\B\R(A, \prec)$ with the $\ell$-subalgebra of $A$ generated by $\R(A,\prec)$.

%\needspace{3\baselineskip}

\begin{definition}\ \label{def: enbal}
\begin{enumerate}
\item For $(A,\prec)\in\pbal$, we call $\R(A,\prec)$ the {\em $\sbal$-skeleton} of $(A,\prec)$, and $(A,\prec)$ is said to be {\em skeletal} provided $A=\B\R(A,\prec)$.
\item Let $\spbal$ be the full subcategory of $\pbal$ consisting of skeletal $(A, \prec)\in\pbal$.
\end{enumerate}
\end{definition}

\begin{theorem} \label{thm: sbal = bnbal}
$\B$ is left adjoint to $\R$, and the two yield an equivalence of $\sbal$ and $\spbal$.
\end{theorem}

\begin{proof}
Let $S \in \sbal$ and $(A, \prec) \in \pbal$. We show that there is a bijection 
\[
\hom_{\sbal}(S, \R(A,\prec)) \to \hom_{\pbal}(\B S, (A, \prec)).
\]
Let $\alpha : S \to \R(A, \prec)$ be an $\sbal$-morphism. By \cref{thm: bal is reflective}, there is a unique $\bal$-morphism
$\B S \to A$ extending $\alpha$, which by \cref{bal morphism being nbal} is a $\pbal$-morphism $\B S \to (A, \prec)$. Define $$\xi_{S, A} = \xi 
: \hom_{\sbal}(S, \R(A,\prec)) \to \hom_{\pbal}(\B S, (A, \prec))$$ 
by letting $\xi(\alpha)$ be the unique $\pbal$-morphism extending $\alpha \in \hom_{\sbal}(S, \R(A, \prec))$. To see that $\xi$ is one-one, let $\xi(\alpha_1) = \xi(\alpha_2)$ for $\alpha_1, \alpha_2 : S \to \R(A, \prec)$. Then
\[
\alpha_1 = \xi(\alpha_1)|_S = \xi(\alpha_2)|_S = \alpha_2. 
\]
To see that $\xi$ is onto, let $\sigma : \B S \to (A, \prec)$ be a $\pbal$-morphism. Set $\alpha = \sigma|_S$. Then $\alpha : S \to (A, \prec)$ is an $\sbal$-morphism by \cref{prop: functor R}. Since
$\xi(\alpha) : \B S \to A$ is the unique $\bal$-morphism extending $\alpha$, we have $\sigma = \xi(\alpha)$. Consequently, $\xi$ is onto, hence a bijection.

For naturality of $\xi$, let $\alpha : S \to T$ be an $\sbal$-morphism and $\sigma : (A, \prec) \to (B, \prec)$ a $\pbal$-morphism. Set $\gamma = \sigma|_{\R(A, \prec)}$. 
%Naturality of $\xi$ follows from the commutativity of the following  diagram, 
We show that the following diagram is commutative,  where $\alpha^*(\delta) =  \delta \circ \alpha$ for $\delta \in \hom_{\sbal}(T, \R(A \prec))$, $(\B \alpha)^*(\tau) = \tau \circ \B \alpha$ for $\tau \in \hom_{\pbal}(\B T, (A, \prec))$, $\gamma_*(\pi) = \gamma \circ \pi$ for $\pi \in \hom_{\sbal}(S, \R(A, \prec))$, and $\sigma_*(\kappa) = \sigma \circ \kappa$ for $\kappa \in \hom_{\pbal}(\B S, (A, \prec))$. 
\[
\begin{tikzcd}
\hom_{\sbal}(T, \R(A \prec)) \arrow[r, "\xi_{T,A}"] \arrow[d, "\alpha^*"'] & \hom_{\pbal}(\B T, (A, \prec)) \arrow[d, "(\B\alpha)^*"] \\
\hom_{\sbal}(S, \R(A, \prec) \arrow[r, "\xi_{S, A}"] \arrow[d, "\gamma_*"'] & \hom_{\pbal}(\B S, (A, \prec)) \arrow[d, "\sigma_*"] \\
\hom_{\sbal}(S, \R(B, \prec) \arrow[r, "\xi_{S, B}"'] & \hom_{\pbal}(\B S, (B, \prec))
\end{tikzcd}
\]
For $\delta \in \hom_{\sbal}(T, \R(A, \prec))$, we have \[
(\B\alpha)^* \xi_{T,A}(\delta) = \xi_{T,A} \circ \B\alpha \quad\textrm{and}\quad \xi_{S, A} \alpha^*(\delta) = \xi_{S, A}(\delta \circ \alpha). 
\]
Since $(\B\alpha)^* \xi_{T,A}(\delta)$ and $\xi_{S, A} \alpha^*(\delta)$ extend $\delta \circ \alpha$, they agree by \cref{thm: bal is reflective}. Similarly, if $\pi \in \hom_{\sbal}(S, \R(A, \prec))$, then \[
\sigma_*\xi_{S,A}(\pi) = \sigma \circ \xi_{S,A}(\pi) \quad\textrm{and}\quad \xi_{S,B}(\gamma_*(\pi)) = \xi_{S,B}(\gamma \circ \pi). 
\]
Since $\gamma = \sigma|_{\R(A, \prec)}$, both $\sigma_*\xi_{S,A}(\pi)$ and $\xi_{S,B}(\gamma_*(\pi))$ extend $\gamma\circ \pi$, so are equal by \cref{thm: bal is reflective}. Consequently, the diagram is commutative.
Thus, $\B$ is left adjoint to $\R$. The essential image of $\B\R$ is $\spbal$,  
and hence $\B$ and $\R$ 
form an equivalence between $\sbal$ and $\spbal$. 
\end{proof}

As an immediate consequence of the above we obtain: 

\begin{corollary}
    $\spbal$ is a coreflective subcategory of $\pbal$.
\end{corollary}

\begin{proof}  
    By the proof of \cref{thm: sbal = bnbal}, $\B$ is fully faithful. Therefore, $\B\R$ is right adjoint to the inclusion functor $\spbal\hookrightarrow\pbal$, and hence $\spbal$ is a coreflective subcategory of $\pbal$ (see, e.g., \cite[Sec.~IV.3]{Mac71}).  
\end{proof}

\section{Nachbin proximities on \texorpdfstring{$\bal$}{bal}-algebras} \label{sec: Nachbin proximities}

As we pointed out, our main examples of $\sbal$-algebras are $C_\le(X)$ and those of proximity $\bal$-algebras are $(C(X),\prec_X)$, where $X$ is a Nachbin space and $\prec_X$ is defined in \cref{eqn: le sub X}. We now show that the assignments $X \mapsto C_\le(X)$ and $X \mapsto (C(X),\prec_X)$ extend to functors.

\begin{proposition} \label{prop: C functor}
There are functors $\C:\Nach\to\pbal$ and $\C_\le:\Nach\to\sbal$ such that $\R\circ\C = \C_\le$.
\[
\begin{tikzcd}
\pbal \arrow[rr, "\R"] && \sbal \\
& \Nach \arrow[ul, "\C"] \arrow[ur, "\C_\le"']
\end{tikzcd}
\]
\end{proposition}

\begin{proof}
We have seen that $\C$ and $\C_\le$ are well defined on objects. For morphisms, let $\varphi : X \to Y$ be a $\Nach$-morphism. Then $\C\varphi : C(Y) \to C(X)$, given by $\C\varphi(f) = f \circ \varphi$, sends $C_\le(Y)$ to $C_\le(X)$ since $f$ and $\varphi$ are order preserving. Because $\C\varphi$ is a $\bal$-morphism, it follows that $\C\varphi$ is a $\pbal$-morphism (see \cref{bal morphism being nbal}) and its restriction to $C_\le(Y)$ is an $\sbal$-morphism. Clearly both $\C$ and $\C_\le$ preserve compositions and identity morphsims.
Therefore, $\C$ and $\C_\le$
are functors, and $\R\circ\C = \C_\le$ since $\R C(X) = C_\le(X)$.
\end{proof}

If $X$ is a Nachbin space, then $\B C_\le(X)$ is a $\bal$-subalgebra of $C(X)$, which in general is properly contained in $C(X)$.

\begin{lemma} \label{BC_le dense in C}
$\B C_\le(X)$ is uniformly dense in $C(X)$.
\end{lemma}

\begin{proof}
By \cref{prop: equivalent conditions for density}, it is sufficient to verify that $\B C_\le(X)$ separates points of $X$. 
Let $x, y \in X$ with $x \ne y$. Since $\le$ is a partial order, we may assume without loss of generality that $x \not\le y$. Because $X \in \Nach$, continuous order-preserving functions on $X$ separate points of $X$ (see \cite[p.~55, Thm.~ 7]{Nac65}). Therefore, there is $f \in C_\le(X)$ with $f(x) > f(y)$.
\end{proof}

This motivates the following:

\begin{definition}
    Let $A\in\bal$ and $\prec$ be a reflexive proximity on $A$. 
    \begin{enumerate}
        \item We call $\prec$ a {\em Nachbin proximity} provided $\B\R(A,\prec)$ is uniformly dense in $A$. 
        \item Let $\nbal$ be the full subcategory of $\pbal$ consisting of $(A,\prec)$, where $\prec$ is a Nachbin proximity.
    \end{enumerate} 
\end{definition}

Before proceeding, we connect the notion of reflexive and Nachbin proximities to other existent notions of proximity on $\bal$-algebras.

\begin{remark} \label{rem: comparison}
    Proximities on $\bal$-algebras were studied in \cite{BMMO15b,BMO16} in connection with de Vries and Gelfand dualities. 
We briefly detail the connection with our new notion of a Nachbin proximity. 
Let $\prec$ be a binary relation on $A\in\bal$.
\begin{enumerate}
\item We call $\prec$ a {\em basic proximity} if $\prec$ satisfies \ref{P1}--\ref{P9}.
\item We call a basic proximity $\prec$ a {\em de Vries proximity} if it satisfies: 
\begin{enumerate}[label=(P\arabic*)]
\setcounter{enumii}{10}
\item \label{P11} If $a \prec b$, then $-b \prec -a$.
\item \label{P12} If $0 < b$, then there is $0 < a$ with $a \prec b$. 
\end{enumerate}
\item We call $\prec$ a {\em Gelfand proximity} if it is a reflexive de Vries proximity. 
\end{enumerate}

The diagram below details the  inclusion relationships between the existing notions of proximity on a $\bal$-algebra. 

\begin{center}
\begin{tikzcd}%[row sep = 3pt]
& \textrm{basic}  \arrow[dr, hookleftarrow] \\
\textrm{de Vries} \arrow[ur, hookrightarrow]  && \textrm{reflexive} \\
\textrm{Gelfand}   \arrow[u, hookrightarrow] \arrow[urr, hookrightarrow] &&\textrm{Nachbin} \arrow[u, hookrightarrow]
\end{tikzcd}
\end{center}
\end{remark}

By \cref{BC_le dense in C},
the $\C_\le$-image of $\Nach$ lands in $\nbal$. 
We next define a functor in the other direction.
To do so, for each $(A,\prec)\in\nbal$, we define a partial order on the space $X_A$ of maximal $\ell$-ideals of $A$.  

\begin{definition} \label{def: le on X}
Let $(A, \prec) \in \nbal$. For $y \in X_A$ set
\[
\thd y = \{ a \in A : \exists c \in y \cap \R(A,\prec)^+ \textrm{ with } |a| \le c\},
\]
and define $\le$ on $X_A$ by $$x \le y \iff \thd y \subseteq x.$$
\end{definition}

In \cref{lem: properties of thd} we prove various properties of $\thd$, for which we use the following lemma about $\bal$-algebras, a proof of which  can be found in \cite[Rem.~2.11]{BMO16}.

\begin{lemma} \plabel{lem: properties of order on XA}
Let $A \in \bal$, $a \in A$, and $x \in X_A$.
\begin{enumerate}
\item \label[lem: properties of order on XA]{a plus or a minus in x} 
$a^+ \in x$ or $a^- \in x$.
\item \label[lem: properties of order on XA]{when is a+x positive} $a^+ \notin x$ iff $a + x > 0 + x$.
\item  \label[lem: properties of order on XA]{when is a+x negative} $a^- \notin x$ iff $a + x < 0 + x$.
\end{enumerate}
\end{lemma}

\begin{lemma} \plabel{lem: properties of thd}
Let $(A, \prec) \in \nbal$ and $ x,y \in X_A$.
\begin{enumerate}
\item \label[lem: properties of thd]{thd an ideal} $\thd y$ is an $\ell$-ideal with $\thd y \subseteq y$.
\item \label[lem: properties of thd]{thd restricts to R} $\thd y \cap \R(A, \prec)^+ = y \cap \R(A, \prec)^+$.
\item \label[lem: properties of thd]{op and idempotent} $\thd$ is order preserving and idempotent.
\item \label[lem: properties of thd]{thd is 1-1} $x \ne y$ implies $\thd x \ne \thd y$.
\end{enumerate}
\end{lemma}

\begin{proof}
\ref{thd an ideal} The inclusion $\thd y \subseteq y$ is clear. We show that $\thd y$ is an $\ell$-ideal. Let $a, b \in \thd y$. Then there are $c, d \in y \cap \R(A, \prec)^+$ with $|a| \le c$ and $|b| \le d$, so $|a + b| \le c + d \in y \cap \R(A, \prec)^+$. Thus, $a+b \in \thd y$. Next, let $a \in \thd y$ and $b \in A$. There is $c \in y \cap \R(A, \prec)^+$ with $|a| \le c$. Since $A$ is bounded, $|b| \le n$ for some $n \in \mathbb{N}$. Therefore, $|ab| \le nc \in y \cap \R(A, \prec)^+$, and hence $ab \in \thd y$. This shows that $\thd y$ is an ideal of $A$. To see that it is an $\ell$-ideal, let $a \in A$ and $b \in \thd y$ with $|a| \le |b|$. There is $c \in y \cap \R(A, \prec)^+$ with $|b| \le c$, yielding that $|a| \le c$, so $a \in \thd y$.

\ref{thd restricts to R} Let $a \in \R(A, \prec)^+$. By definition, $a \in y$ iff $a \in \thd y$. The result follows.

\ref{op and idempotent} It is clear that $\thd$ is order preserving. For idempotent, we have $\thd(\thd y) \subseteq \thd y$ by \ref{thd an ideal}. 
For the reverse inclusion, if $a \in \thd y$, there is $c \in y \cap \R(A, \prec)^+$ with $|a| \le c$. By \ref{thd restricts to R}, $c \in \thd y$, and hence $a \in \thd(\thd y)$.

\ref{thd is 1-1} 
Since $(A,\prec)\in\nbal$, $\B\R(A,\prec)$ is uniformly dense in $A$. 
Therefore, $\B\R(A,\prec)$ separates points of $X_A$. So, as $x \ne y$, there are $c, d \in \R(A, \prec)$ with $c-d \in x$ and $c-d \notin y$.  Because $A/x \cong \mathbb{R}$, there is $r \in \mathbb{R}$ with $c-r \in x$. We then have $d-r \in x$ and either $c-r \notin y$ or $d-r \notin y$. 
We only consider the case when $c-r \notin y$ since the case when $d-r \notin y$ is proved similarly. 
If $s \in \mathbb{R}$ with $c -s \in y$, then $r \ne s$. If $r < s$ then $(c-r)^+ \notin y$ since $(c-r) + y > 0 + y$; and if $r > s$ then $(c-s)^+ \notin x$ because $(c-s) + x > 0 + x$ (see \cref{when is a+x positive}). Also, by \cref{rem: proximity diagram}, $c - r \in x$ implies $(c - r)^+ \in x \cap \R(A,\prec)^+$ and $c - s \in y$ gives $(c - s)^+ \in y \cap \R(A,\prec)^+$. Thus, if $r<s$ then $(c-r)^+ \in \thd x - \thd y$, and if $r>s$ then $(c-s)^+ \in \thd y - \thd x$, yielding that $\thd x \ne \thd y$.
\end{proof}

\begin{theorem} \label{thm: X_A is a Nachbin space}
If $(A, \prec) \in \nbal$, then $(X_A, \le) \in \Nach$. 
\end{theorem}

\begin{proof}
That $\le$ is a partial order follows from \cref{lem: properties of thd}.
It remains to show that $\le$ is a closed order. For this, it is sufficient to show that if $x \not\le y$, then 
$x$ and $y$ can be separated by disjoint open sets $U,V$ such that $U$ is an upset and $V$ is a downset.
From $x \not\le y$ it follows that $\thd y \not\subseteq x$, so there is $c \in \R(A, \prec)^+$ with $c \in y - x$. Then  
$c > 0$. Because $A/x \cong \mathbb R$, there is $0 < r \in \mathbb{R}$ with $c - r \in x$. Choose $s \in \mathbb R$ with $0 < s < r$. 
Then $0+x = (c-r) + x < (c-s) + x$. 
so $(c-s)^+ \notin x$ by \cref{when is a+x positive}. Also, $(c-s) +y < 0 + y$ since $c \in y$ and $0 < s$. Therefore, $(c-s)^- \notin y$ by \cref{when is a+x negative}.  
Set
\[
U = \{ z \in X_A : (c-s)^+ \notin z\} \quad\textrm{and}\quad V = \{ z \in X_A : (c - s)^- \notin z\}. 
\]
Recalling the topology from \cref{sec: preliminaries}, $U, V$ are open. Moreover, $x \in U$, $y \in V$, and $U \cap V = \varnothing$ by \cref{a plus or a minus in x}. To see that $U$ is an upset, suppose that $z \in U$ and $z \le z'$. Then $\thd z' \subseteq z$ and $(c-s)^+ \notin z$. Because $(c-s)^+ \in \R(A, \prec)^+$ (see \cref{rem: proximity diagram}),  
we see that $(c-s)^+ \notin z'$, so $z' \in U$.

To see that $V$ is a downset, let $z' \in V$ and $z \le z'$. Then $\thd z' \subseteq z$ and $(c-s)^- \notin z'$. We have $c + z' < s + z'$ by \cref{when is a+x positive}. Since $A/z' \cong \mathbb R$, there is $t \in \mathbb{R}$ with $c + z' = t + z'$. Therefore, $t + z' < s + z'$, so $t < s$. Moreover, $(c-t)^+ \in z' \cap \R(A, \prec)^+$, so $(c-t)^+ \in \thd z' \subseteq z$. This yields $c + z \le t + z$ by \cref{when is a+x positive}. Thus, $c + z 
< s + z$ because $t<s$, so $(c-s)^- \notin z$ by \cref{when is a+x negative}. Consequently, $z \in V$, and hence $V$ is a downset.
\end{proof}

\begin{remark} \label{rem: quasiorder vs partial order}
In the above proof, the only time we used that $(A, \prec) \in \nbal$ was to show that $\le$ is antisymmetric. Thus, if $(A, \prec) \in \pbal$, then $\le$ is a closed quasi-order on $X_A$. 
However, there exist $(A, \prec) \in \pbal$ such that $\le$ is not a partial order. 
For example, let $A = \mathbb{R}^2$  
and define $\prec$ on $A$ by $(a,b) \prec (c,d)$ if there is $r \in \mathbb{R}$ with $a \le r \le c$ and $b \le r \le d$. Then $(A,\prec)\in\pbal$, 
but $\B\R(C(X), \prec) = \{(r,r) : r \in  \mathbb{R}\}$ is not uniformly dense in $A$. We have 
$X_A = \{ x,y \}$, where $x = \mathbb{R} \times \{0\}$ and $y = \{0\} \times \mathbb{R}$. It is straightforward to see that $\thd x = \thd y = \{0\}$, and hence 
$x\le y$ and $y \le x$. Thus, $\le$ is not a partial order.
\end{remark}

We extend the functor $\X : \bal \to \KHaus$ to a functor from $\nbal$ to $\Nach$. For this we need the following lemma. For each $(A,\prec) \in \nbal$, we denote by $\iota_A$ the inclusion $\B \R (A,\prec) \hookrightarrow A$ (see \cref{con: BS inside A}).

\begin{lemma} \label{lem: PR to 1 is a bimorphism}
Let $(A, \prec) \in \nbal$. The  map $\iota_A : \B\R(A, \prec) \to (A, \prec)$ is an $\nbal$-morphism whose dual $\X\iota_A : \X(A, \prec) \to \X\B\R(A, \prec)$ is a $\Nach$-isomorphism.
\end{lemma}

\begin{proof}
For ease of notation we set $(A_0, \prec_0) = \B\R(A, \prec)$. Note that ${\R(A_0, \prec_0) = \R(A, \prec)}$. Because $A_0$ is a $\bal$-subalgebra of $A$, the map $\iota_A$ is a $\bal$-morphism and, since $\prec_0$ is the restriction of $\prec$ to $A_0$, $\iota_A$ is an $\nbal$-morphism. As $A_0$ is uniformly dense in $A$, the induced map $\X\iota_A : X_A \to X_{A_0}$ is a homeomorphism by \cref{prop: equivalent conditions for density}. It suffices to show that $x \le y$ iff ${\X\iota_A(x) \le \X\iota_A(y)}$ for each $x, y \in X_A$. 
We have $\X\iota_A(x) = \iota_A^{-1}(x) = x \cap \B\R(A, \prec)$. Moreover, given $x, y \in X_A$, we have $x \le y$ iff $\thd y \subseteq x$, which is equivalent to $y \cap \R(A, \prec)^+ \subseteq x$.  The reflexive elements of $(A, \prec)$ and $\B\R(A, \prec)$ are the same. Therefore, 
\begin{align*}
x \le y &\iff y \cap \R(A, \prec)^+  \subseteq x \\
&\iff \left(y \cap \B\R(A, \prec)\right) \cap \R(A, \prec)^+ \subseteq x \cap \B\R(A, \prec) \\
&\iff \iota_A^{-1}(x) \le \iota_A^{-1}(y).\qedhere
\end{align*}
\end{proof}

\begin{proposition} \label{prop: functor X}
There is a functor $\X : \nbal \to \Nach$ such that $\R\circ\C = \C_\le$ and $\Y\circ\R$ is naturally equivalent to $\X$,
where $\Y$ is the composition $\X\B$.
\[
\begin{tikzcd}
\nbal \arrow[rr, "\R"] \arrow[dr, shift left = .5ex, "\X"] && \sbal \arrow[dl, shift right = .5ex, "\Y"']\\
& \Nach \arrow[ur, shift right = .5ex, "\C_\le"'] \arrow[ul, shift left = .5ex, "\C"] &
\end{tikzcd}
\]
\end{proposition}

\begin{proof}
With each $(A, \prec) \in \nbal$ we associate $(X_A,\le)$ and with each  
$\nbal$-morphism $\alpha : A \to B$ we associate $\varphi := \alpha^{-1}$. That $(X_A, \le) \in \Nach$ follows from \cref{thm: X_A is a Nachbin space}, and that $\varphi$ is continuous follows from Gelfand duality (see \cref{sec: preliminaries}).
We show that $\varphi$ is order preserving. 
Let $x, y \in X_B$ with $x \le y$, so $\thd y \subseteq x$. If $a \in \thd \varphi(y)$, then there is $0 \le c \in \varphi(y) \cap \R(A, \prec)^+$ with $|a| \le c$. Since
\[
|\alpha(a)| = \alpha(a) \vee -\alpha(a) = \alpha(a \vee -a) = \alpha(|a|), 
\]
we have
$$|\alpha(a)| = \alpha(|a|) \le \alpha(c) \in \alpha[\R(A,\prec)] \subseteq \R(B, \prec),$$ so $\alpha(a) \in x$, and hence $a \in \varphi(x)$. Therefore, $\varphi(x) \le \varphi(y)$, and so $\varphi$ is a $\Nach$-morphism. Thus, $\X$ is well defined on both objects and morphisms, and it is straightforward to see that 
$\X$ preserves composition and identities. 

We have $\R\circ\C = \C_\le$ by \cref{prop: C functor}. 
To see that $\Y\circ\R$ is naturally equivalent to $\X$, let $(A, \prec) \in \nbal$. Then $\X\iota_A : \X(A, \prec) \to \X\B\R(A, \prec)$ is a $\Nach$-isomorphism by \cref{lem: PR to 1 is a bimorphism}. For naturality, it is sufficient to show that the diagram below commutes, where
$\alpha : (A, \prec) \to (B, \prec)$ is an $\nbal$-morphism.
\[
\begin{tikzcd}
\X(B, \prec) \arrow[r, "\X\iota_B"] \arrow[d, "\X\alpha"']& \Y\R(A, \prec) \arrow[d, "\Y\R\alpha"] \\
\X(A, \prec) \arrow[r, "\X\iota_A"'] & \Y\R(A, \prec)
\end{tikzcd}
\]
We have $\Y\R\alpha \circ \X\iota_B = \X(\iota_B \circ \B\R\alpha)$ and $\X\iota_A \circ \X\alpha = \X(\alpha\circ \iota_A)$. It then suffices to show that $\iota_B \circ \B\R\alpha = \alpha \circ \iota_A$, which follows from \cref{thm: bal is reflective} since both $\alpha \circ \iota_A$ and $\iota_B \circ \B\R\alpha$ extend $\alpha|_{\R(A, \prec)} : \R(A, \prec) \to B$. 
\[
\begin{tikzcd}
\B\R(A, \prec) \arrow[r, "\iota_A"] \arrow[d, "\B\R\alpha"'] & A \arrow[d, "\alpha"] \\
\B\R(B, \prec) \arrow[r, "\iota_B"'] & B
\end{tikzcd}
\]
\end{proof}

We finish the section by showing that the functors $\X,\C$ and $\Y,\C_\le$ form a pair of contravariant adjunctions. 

\begin{theorem}\ \plabel{thm: contravariant adjunction}
\begin{enumerate}
\item \label[thm: contravariant adjunction]{nbal case} The functors $\C : \Nach \to \nbal$ and $\X : \nbal \to \Nach$ form a contravariant adjunction with the units $\phi : 1_{\nbal} \to \C\X$ and $\eta : 1_\Nach \to \X\C$ such that $\eta$ is a natural isomorphism.
\item \label[thm: contravariant adjunction]{sbal case} The functors $\C_\le : \Nach \to \sbal$ and $\Y : \sbal \to \Nach$ form a contravariant adjunction with the units $\psi : 1_{\sbal} \to \C_\le\Y$ and $\zeta : 1_\Nach \to \Y\C_\le$ such that $\zeta$ is a natural isomorphism.
\end{enumerate}
\end{theorem}

\begin{proof}
\ref{nbal case} By Gelfand duality, $\X : \bal \to \KHaus$ and $\C : \KHaus \to \bal$ form a contravariant adjunction with the units $\phi : 1_{\bal} \to \C\X$ and $\eta : 1_\KHaus \to \X\C$ such that $\eta$ is a natural isomorphism. We show that they lift to the units $\phi : 1_{\nbal} \to \C\X$ and $\eta : 1_\Nach \to \X\C$ yielding a contravariant adjunction $\X : \nbal \to \Nach$ and $\C : \Nach \to \nbal$ such that $\eta$ is a natural isomorphism. For this, it suffices to show that $\phi_A$ is an $\nbal$-morphism for each $(A, \prec) \in \nbal$ and that $\eta_X$ is a $\Nach$-isomorphism for each $(X, \le) \in \Nach$. 

Let $(A, \prec) \in \nbal$. To see that $\phi_A$ is an $\nbal$-morphism, it suffices to show that it preserves proximity. Let $a, b \in A$ with $a \prec b$. Then there is $c \in \R(A, \prec)$ with $a \le c \le b$. To show that $\phi_A(a) \prec_X \phi_A(b)$, it suffices to show that $\phi_A(c) \in C_\le(X_A)$. Let $x, y \in X_A$ with $x \le y$. Then $\thd y \subseteq x$. Set $r = \phi_A(c)(y)$. Then $c-r \in y$, so $(c-r)^+ \in \thd y$ since $(c-r)^+$ is reflexive, which yields $(c-r)^+ \in x$. Therefore, $(\phi_A(c)-r)^+(x) = 0$. 
This yields $(\phi_A(c)-r)(x) \le 0$, so $\phi_A(c)(x) \le r = \phi_A(c)(y)$, and thus $\phi_A(c) \in C_\le(X_A)$.

Let $(X, \le) \in \Nach$. To see that $\eta_X$ is a $\Nach$-isomorphism, suppose that $x \le y$. If $f \in \thd M_y$, then there is $c \in C_\le(X)$ with $|f| \le c$ and $c(y) =0$. Since $c(x) \le c(y)$, $$0 \le |f(x)| \le c(x) \le c(y) = 0,$$ so $f(x) = 0$, and hence $f \in M_x$. This shows that $\thd M_y \subseteq M_x$. Conversely, suppose that ${\thd M_y \subseteq M_x}$. Since $X$ is a Nachbin space, if $x \not\le y$, there is $c \in C_\le(X)$ with $c(x) > c(y)$ (see the proof of \cref{BC_le dense in C}). Letting $r = c(y)$, we have $(c-r)(y) = 0$, so $(c - r)^+(y) = 0$. Therefore, $(c - r)^+ \in \thd M_y$, and hence $(c - r)^+ \in M_x$. This implies that $(c - r)^+(x) = 0$, so $(c-r)(x) \le 0$, and hence $c(x) \le r = c(y)$. The obtained contradiction shows that $x \le y$. Thus, $\eta_X$ is a $\Nach$-isomorphism.

Consequently, the natural bijection $\theta : \hom_{\bal}(A, \C X) \to \hom_{\KHaus}(X, \X A)$ lifts to the natural bijection $\theta : \hom_{\nbal}(A, \C X) \to \hom_{\Nach}(X, \X A)$. 

\ref{sbal case} Define $\psi : 1_{\sbal} \to \C_\le\Y$ by $\psi_S = \phi_{\B S} \circ \varepsilon_S$ for each $S \in \sbal$. By \ref{nbal case} and \cref{varepsilon}, each $\psi_S$ is well defined.
Also, define  $\zeta : 1_{\Nach} \to \Y\C_\le$ by $\zeta_X = \X\iota_{C(X)}\circ \eta_X$ for each $X \in \Nach$. By \ref{nbal case} and \cref{lem: PR to 1 is a bimorphism}, each $\zeta_X$ is a well-defined $\Nach$-isomorphism.
For each $S \in \sbal$ and $X \in \Nach$, consider the following diagram. 

\[
\begin{tikzcd}
\hom_{\nbal}(\B S, \C X) \arrow[r, "\theta"] & \hom_{\Nach}(X, \X \B S) \arrow[d, equal] \\
\hom_{\sbal}(S, \C_\le X) \arrow[u, "\nu"] \arrow[r, "\tau"'] & \hom_{\Nach}(X, \Y S)
\end{tikzcd}
\]
Since $\Y = \X\B$, we see that $\hom_{\Nach}(X, \X \B S) = \hom_{\Nach}(X, \Y S)$. Moreover, $\theta$ is a bijection by \ref{nbal case}. Furthermore, $\nu : \hom_{\sbal}(S, \C_\le X) \to \hom_{\nbal}(\B S, \C X)$ is defined by $\nu(\alpha) = \iota_{C(X)} \circ \B\alpha$ for each $\alpha \in \hom_{\sbal}(S, \C_\le X)$. Finally, $\tau : \hom_{\sbal}(S, \C_\le X) \to \hom_{\Nach}(X, \Y S)$ is given by $\tau = \theta\circ \nu$. 

We show that $\nu$ is a bijection. To see that $\nu$ is one-to-one, let $\alpha, \beta : S \to C_\le(X)$ be $\sbal$-morphisms with $\nu(\alpha) = \nu(\beta)$. Then $\B\alpha = \B\beta$ since $\iota_{C(X)}$ is one-to-one. Thus, $\alpha = \beta$ by \cref{thm: sbal = bnbal}.
To show that $\nu$ is onto, let $\sigma : \B S \to \C X$ be an $\nbal$-morphism. Then the restriction $\alpha : S \to C_\le(X)$ is a well-defined $\sbal$-morphism, so $\B\alpha : \B S \to \B\C_\le X$ is an $\nbal$-morphism. Both $\sigma$ and $\iota_{C(X)} \circ \B\alpha$ extend $\alpha$, hence they are equal. Therefore, $\sigma = \nu(\alpha)$. Thus, $\nu$ is a bijection, yielding that so is $\tau$. Naturality of $\nu$ follows by \cref{thm: sbal = bnbal}, which together with the naturality of $\theta$ yields the naturality of $\tau$. Consequently, $\C_\le : \Nach \to \sbal$ and $\Y : \sbal \to \Nach$ form a contravariant adjunction with the units $\psi : 1_{\nbal} \to \C_\le\Y$ and $\zeta : 1_{\Nach} \to \Y\C_\le$ such that $\zeta$ is a natural isomorphism.
\end{proof}

While $\phi$ is not a natural isomorphism, in the next section we will identify a full subcategory of $\nbal$ for which it is a natural isomorphism, thus extending Gelfand duality from $\KHaus$ to $\Nach$. This requires an appropriate generalization of the Stone-Weierstrass theorem and Dieudonn\'{e}'s lemma, which we next turn to.

\section{Lifting Gelfand duality to the category of Nachbin spaces}\label{sec: duality}

In this section we prove the main results of this paper. We introduce the full subcategories $\unbal$ of $\nbal$ and $\usbal$ of $\sbal$ and prove that the contravariant adjunctions of \cref{thm: contravariant adjunction} restrict to dual equivalences between $\Nach$ and $\unbal$ and between $\Nach$ and $\usbal$. As a consequence, we obtain that $\unbal$ is equivalent to $\usbal$. 
Our main machinery is an appropriate generalization of two classic results to Nachbin spaces: the Stone-Weierstrass theorem and Dieudonn\'{e}'s lemma. In addition, we introduce a full subcategory $\uspbal$ of $\spbal$ and show that it is also dually equivalent to $\Nach$. We conclude the section with a diagram detailing all the equivalences and dual equivalences established in this paper.

We start with a generalization of the Stone-Weierstrass theorem to Nachbin spaces. A similar result was proved in \cite[Thm.~4.3]{DH18} using an analytic argument. We give a more lattice-theoretic proof below, which is based on \cite{BMO20d} (see 
Lemmas 3, 13, and 15).

\begin{theorem}[Stone-Weierstrass for $\Nach$] \label{prop: SW for sbal}
 Let $(A, \prec) \in \nbal$. Then $\phi_A[\R(A, \prec)]$ is uniformly dense in $C_\le(X_A)$.
\end{theorem}

\begin{proof}
For simplicity, we identify $A$ with $\phi_A[A] \subseteq C(X_A)$.
Let $f \in C_\le(X_A)$. 
If $f$ is constant, then $f \in \R(A, \prec)$, so we assume that $f$ is not constant.
Since $X_A$ is compact, $f$ is bounded. We let $s = \sup f[X]$, ${t = \inf f[X]}$, and observe that $f$ attains these values.   
Moreover, since $f$ is not constant, $t<s$. For each $r$ with $t < r \le s$ set $F_r = f^{-1}[r, \infty)$. Then $F_r$ is a proper nonempty closed upset of $X_A$.
We produce $a_{r, y} \in \R(A, \prec)$ for each 
$y \notin F_r$ satisfying
\begin{itemize}
    \item $r \le a_{r,y} \le s$.
    \item $a_{r,y}(y) = r$.
    \item $a_{r, y}(x) = s$ for each $x \in F_r$.
\end{itemize}

Fix $r$ and $y \notin F_r$. For each $x \in F_r$, we have $x \not\le y$ because $F_r$ is an upset of $X_A$. 
Therefore, $\thd y \not\subseteq x$. If $a \in \thd y - x$, then there is $c \in y \cap \R(A, \prec)^+$ with $|a| \le c$ and $c \notin x$, so $c(x) > 0$ and $c(y) = 0$. Multiplying $c$ by $c(x)^{-1}$ yields $c_x \in \R(A, \prec)^+$ with $c_x(x) = 1$ and $c_x(y) = 0$. 
This gives that $\{ c_x^{-1}(1/2, \infty) : x \in F_r\}$ is an open cover of $F_r$, so compactness of $X_A$ implies that there are $x_1, \dots, x_n \in F_r$ with $F_r \subseteq c_{x_1}^{-1}(1/2, \infty) \cup \dots \cup c_{x_n}^{-1}(1/2, \infty)$. Set $b_y = 2(c_{x_1} \vee \dots \vee c_{x_n})$. Then $b_y \ge 1$ on $F_r$ and $b_y(y) = 0$. Finally, set $a_{r, y} = r + (s - r)(b_y \wedge 1)$. The $a_{r, y}$ satisfy the properties above. 

We next show that $f(x) = \inf\{ a_{r,y}(x) : t < r \le s, y \notin F_r\}$ for each $x \in X_A$. Let $x \in X_A$. To see that $f(x) \le a_{r, y}(x)$, if $x \in F_r$, then $f(x) \le s = a_{r, y}(x)$. If $x \notin F_r$, then $f(x) < r \le a_{r, y}(x)$. Thus, $f(x) \le a_{r, y}(x)$ for each $r, y$.
Let $\varepsilon > 0$  
and set $r=f(x)$. If $r = s$, then $f(x) = a_{s,y}(x)$ for any $y \notin F_s$. If $r < s$, set $r' = \min\{r+\varepsilon, s\}$, so $x \notin F_{r'}$, and hence $a_{r', y}(x) = r'$ for any $y \notin F_{r'}$. 
Therefore, $f(x) = \inf\{ a_{r,y}(x) : t < r \le s, y \notin F_r\}$.

To finish the argument, let $\varepsilon > 0$. By the above, there is $S \subseteq \R(A, \prec)$ with $f(x) = \inf \{ a(x) : a \in S\}$ for each $x \in X_A$. Therefore, 
there is $a_x \in S$ with $a_x(x) < f(x) + \varepsilon$. Let $U_x = (f+\varepsilon - a_x)^{-1}(0, \infty)$, an open neighborhood of $x$. Compactness yields $x_1, \dots, x_n \in X_A$ with $X_A = U_{x_1} \cup \dots \cup U_{x_n}$. Set $a = a_{x_1} \wedge \dots \wedge a_{x_n}$ and observe that $a \in \R(A,\prec)$ because each $a_{x_i} \in \R(A,\prec)$. 
We have $f(x) \le a(x)$ since $f(x) \le a_{x_i}(x)$ for each $i$. Also, 
$x \in U_{x_i}$ for some $i$, so $a(x) \le a_{x_i}(x) < f(x) + \varepsilon$. Thus, $|f-a|(x) = |f(x) - a(x)| \le \varepsilon$ for all $x \in X_A$, and so $\|f-a\| \le \varepsilon$.
\end{proof}

We next obtain a version of Dieudonn\'{e}'s lemma for $\nbal$-algebras. Our proof follows along the lines of \cite[Lem.~5]{BMO20d}. 

\begin{theorem}[Dieudonn\'{e} for $\nbal$] \label{prop: Dieudonne}
Suppose $(A, \prec) \in \nbal$ and identify $(A,\prec)$ with its image in $(C(X_A),\prec_{X_A})$. Let $S$ be the closure of $\R(A, \prec)$ in $C(X_A)$, and $\prec'$ the closure of $\prec$ in $C(X_A) \times C(X_A)$. 
Then $S\in\sbal$ and ${\prec'} = {\prec_S}$. 
\end{theorem}

\begin{proof}
By \cref{prop: functor R},  $\R(A,\prec)\in\sbal$. It is well known (and straightforward to see) that all operations on $C(X_A)$ are continuous. Therefore, $S$ is an $\sbal$-subalgebra of $C(X_A)$, and hence $S \in \sbal$.  We show that ${\prec'} = {\prec_S}$.

First, suppose that $f \prec_S g$. Then there is $s \in S$ with $f \le s \le g$. By assumption, there is a Cauchy sequence $\{a_n\}$ from $\R(A, \prec)$ converging to $s$. Let $\| s - a_n\| = r_n$. Then $\lim r_n = 0$ and $a_n - r_n \le s \le a_n + r_n$.  By Gelfand duality, there is a sequence $\{f_n\}$ from $A$ converging to $f$, so $f_n - r_n' \le f \le f_n + r_n'$ where $r_n' = \| f - f_n\|$. By replacing $f_n$ with $f_n - r_n'$, we may assume that $f_n \le f$ for each $n$. Similarly, there is a sequence $\{g_n\}$ from $A$ 
converging to $g$ and  
we may assume that $g \le g_n$ for each $n$. Therefore,
\[
f_n - r_n \le f - r_n \le s - r_n \le a_n \le s + r_n \le g + r_n \le g_n + r_n.
\]
Thus, $f_n -r_n \prec g_n + r_n$. Because $\lim(f_n - r_n, g_n + r_n) = (f, g)$, 
we get $f \prec' g$. 

Next, suppose that $f \prec' g$. By the definition of $\prec'$, $(f, g) = \lim (f_n, g_n)$ with $f_n \prec g_n$ for each $n$.
Note that $f \le g$ because $f_n \le g_n$ for each $n$ and $\le$ is closed in $C(X_A) \times C(X_A)$. By induction we construct a sequence $\{a_n : n \ge 0\}$ in $\R(A, \prec)$ such that for each $n \ge 1$,
\begin{align}
f - 1/2^n &\le a_n \le g  \label{eqn4}\\
a_{n-1} - 1/2^{n-1} &\le a_{n} \le a_{n-1} + 1/2^{n-1}. \label{eqn5}
\end{align}
We utilize the following:

\begin{claim} \label{claim}
Let $\{b_n\}$ and $\{c_n\}$ be sequences in $A$ such that $b_n$ converges to $b$ and $c_n$ converges to $c$, where $b,c \in C(X_A)$. If $b_n \prec c_n$ for each $n$, then for each $0 < r \in \mathbb{R}$ there is $a \in \R(A, \prec)$ with $b - r \le a \le c$.
%For each $0 < r \in \mathbb{R}$ there is $a \in \R(A, \prec)$ with $f - r \le a \le g$.
\end{claim}

\begin{proofclaim}
Because $\{b_n\}$ converges to $b$, there is $n$ such that $\|b - b_n \| \le r/2$. Therefore, $b - r/2 \le b_n \le b + r/2$. Similarly, there is $m$ such that $\|c - c_m \| \le r/2$, so $c - r/2 \le c_m \le c + r/2$. Choose $p \ge \max\{n,m\}$. Then $b - r/2 \le b_p \prec c_p \le c + r/2$. Since $b_p \prec c_p$, there is $a' \in \R(A, \prec)$ with $b_p \le a' \le c_p$. Thus, $b - r/2 \le a' \le c + r/2$. Setting $a = a' - r/2$, we have $a \in \R(A, \prec)$ and $b - r \le a \le c$. \qed
\end{proofclaim} 

For the base case, by \cref{claim} there is $a_1 \in \R(A, \prec)$ with $f - 1/2 \le a_1 \le g$. Set $a_0 = a_1$. Then (\ref{eqn4}) and (\ref{eqn5}) are satisfied for $n = 1$. Suppose that $m \ge 1$ and we have
$a_0, \dots, a_m \in \R(A, \prec)$ satisfying (\ref{eqn4}) and (\ref{eqn5}) for all $1 \le n \le m$. By (\ref{eqn4}) for $n = m$ we get $f \le a_m  + 1/2^m$. In addition, it is clear that $a_{m} - 1/2^{m+1} \le a_m  + 1/2^m$. Therefore,
\[
f \vee (a_{m} - 1/2^{m+1}) \le a_m  + 1/2^m.
\]
Since $f, a_m\le g$, it is also clear that $f \vee (a_{m} - 1/2^{m+1}) \le g$. Thus, 
\[
f \vee (a_{m} - 1/2^{m+1}) \le g \wedge (a_m  + 1/2^m).
\]
By \ref{S2},
\begin{align*}
\left[(f - 1/2^{m+1}) \vee (a_{m} - 1/2^m)\right] + 1/2^{m+1} &=&  \\
(f - 1/2^{m+1} + 1/2^{m+1}) \vee (a_{m} - 1/2^m + 1/2^{m+1}) &=& \\
f \vee (a_{m} -1/2^{m+1}).
\end{align*}
Consequently,
\[
\left[(f - 1/2^{m+1}) \vee (a_{m} - 1/2^m)\right] + 1/2^{m+1} \le g \wedge (a_m + 1/2^m).
\]
By \cref{claim}, there is $a_{m+1} \in \R(A, \prec)$ satisfying
\[
(f - 1/2^{m+1}) \vee (a_{m} - 1/2^m) \le a_{m+1} \le g \wedge (a_m + 1/2^m).
\]
Therefore,
\begin{align*}
f - 1/2^{m+1} &\le a_{m+1} \le g \\
a_{m} - 1/2^m &\le a_{m+1} \le a_m + 1/2^m.
\end{align*}
Thus, (\ref{eqn4}) and (\ref{eqn5}) hold for $n = m+1$. By induction we have produced the desired sequence. Equation~(\ref{eqn5}) implies that $\{a_n\}$ is a
Cauchy sequence in $\R(A, \prec)$, so has a uniform limit $a \in S$. Taking limits in (\ref{eqn4}) as $n \to \infty$ gives $f \le a \le g$. Consequently, $f \prec_S g$. 
\end{proof}

The above two theorems motivate the following: 

\begin{definition} \plabel{def: unbal and usbal}
\hfill
\begin{enumerate}
    \item \label[def: unbal and usbal]{unbal} Let $\unbal$ be the full subcategory of $\nbal$ consisting of those $(A, \prec)$ for which
    $A$ is uniformly complete and $\prec$ is a closed subset of $A \times A$. 
    \item \label[def: unbal and usbal]{usbal} Let $\usbal$ be the full subcategory of $\sbal$ consisting of uniformly complete $\sbal$-algebras.
\end{enumerate}
\end{definition}

The next proposition gives a convenient characterization of when a Nachbin proximity on a uniformly complete $A \in \bal$ is closed. 

\begin{proposition} \label{rem: consequences of Dieudonne}
Suppose that $(A, \prec)\in \nbal$ with $A$ uniformly complete. Then $(A,\prec) \in \unbal$ iff $\R(A, \prec)$ is closed in $A$.
\end{proposition}

\begin{proof}
That $\R(A, \prec)$ closed in $A$ implies that $\prec$ is closed in $A \times A$ is immediate from  \cref{prop: Dieudonne}. For the other implication, let $\prec$ be closed in $A\times A$. If $a \in A$ is in the closure of $\R(A, \prec)$, 
there is a sequence $\{a_n\}$ from $\R(A, \prec)$ converging to $a$. Thus, $\{(a_n, a_n)\}$ converges to $(a,a)$. Since $\prec$ is closed in $A \times A$ and $a_n \prec a_n$ for each $n$, we conclude that $a \prec a$, so $a \in \R(A, \prec)$, and hence $\R(A, \prec)$ is closed in $A$.
\end{proof}

We show that not every Nachbin proximity on a uniformly complete $\bal$-algebra is closed.

\begin{example} \label{ex: PP example} 
Let $X = [0,1]$, $A = C(X)$, and $PP(X)$ be the set of piecewise polynomial functions in $C(X)$. Since $PP(X)$ is a $\bal$-subalgebra of $C(X)$ (see e.g.~\cite[Ex.~2.5.2]{BMO13a}) and separates points of $X$, $PP(X)$ is uniformly dense in $C(X)$ by \cref{prop: equivalent conditions for density}. Define $\prec$ on $A$ by $f \prec g$ if there is ${c \in C_\le(X) \cap PP(X)}$ with $f \le c \le g$. 
Since $C_\le(X) \cap PP(X)$ is an $\sbal$-subalgebra of $PP(X)$, $\prec$ is a reflexive proximity on $A$. Because each $f \in PP(X)$ has bounded variation, $PP(X) = \B\R(A, \prec)$ (see \cite[p.~117]{RF10}). 
Therefore, $\prec$ is a Nachbin proximity. 
However, $\prec$ is not a closed proximity. To see this, 
let $f \in C(X)$ be given by $f(x) = e^x$. Since there is a sequence $\{p_n\}$ in $PP(X)$ converging to $f$,  $\{(p_n, p_n)\}$ 
converges to $(f,f)$ in $C(X) \times C(X)$. But $f \not\prec f$ since $f \notin PP(X)$. 
Therefore, $\prec$ is not closed. 
Observe that 
$\phi_A : (A, \prec) \to (C(X_A), \prec_{X_A})$ preserves but does not reflect proximity. Indeed, we have that $\phi_A(f) \prec_{X_A} \phi_A(f)$ since $\phi_A(f) \in C_\le(X_A)$, but $f \not\prec f$.
\end{example}

Standard examples of $\usbal$-algebras and $\unbal$-algebras come from Nachbin spaces. Indeed, if $X\in\Nach$, then $C_\le(X)\in\usbal$ since the limit of any Cauchy sequence from $C_\le(X)$ is again in $C_\le(X)$ because $\le$ is closed in $X\times X$. Therefore, $C_\le(X)$ is closed in $C(X)$, and hence $(C(X),\prec_X) \in \unbal$ by \cref{BC_le dense in C,rem: consequences of Dieudonne}. We are ready to prove that these examples are typical by establishing dual equivalence between $\Nach$, $\unbal$, and $\usbal$.

\begin{theorem} \label{thm: main duality}
$\unbal$ is equivalent to $\usbal$ and dually equivalent to $\Nach$.
\[
\begin{tikzcd}
\unbal \arrow[rr, "\R"] \arrow[dr, shift left = .5ex, "\X"] && \usbal \arrow[dl, shift right = .5ex, "\Y"']\\
& \Nach \arrow[ur, shift right = .5ex, "\C_\le"'] \arrow[ul, shift left = .5ex, "\C"] &
\end{tikzcd}
\]
\end{theorem}

\begin{proof}
As we observed before the theorem, we have 
$(C(X),\prec_X) \in \unbal$ and $C_\le(X) \in \usbal$ for each $X \in \Nach$. Thus, the functors $\C : \Nach \to \unbal$ and $\C_\le : \Nach \to \usbal$ are well defined (recall \cref{prop: C functor}). Let $(A, \prec) \in \unbal$. To see that $\R(A, \prec) \in \usbal$, let $\{a_n\}$ be a Cauchy sequence in $\R(A, \prec)$. Then $a = \lim a_n \in A$ since $A$ is uniformly complete. Because $\prec$ is closed in $A \times A$, it follows that $a \in \R(A, \prec)$ (see \cref{rem: consequences of Dieudonne}). Thus, $\R(A, \prec) \in \usbal$, and hence $\R$ restricts to a functor from $\unbal$ to $\usbal$. 

To see that $\unbal$ and $\Nach$ are dually equivalent, by \cref{nbal case} it suffices to show that $\phi : 1_{\unbal} \to \C\X$ is a natural 
isomorphism. If $(A, \prec) \in \unbal$, then $\phi_A : A \to C(X_A)$ is a $\bal$-isomorphism by Gelfand duality and an $\nbal$-morphism by \cref{nbal case}. Thus, it is left to show that $\phi_A$ reflects proximity. Let $a,b \in A$ with $\phi_A(a) \prec_{X_A} \phi_A(b)$. Then there is $f \in C_\le(X_A)$ with $\phi_A(a) \le f \le \phi_A(b)$. By \cref{prop: SW for sbal}, $\phi_A[\R(A, \prec)]$ is uniformly dense in $C_\le(X_A)$. The previous paragraph shows that $\phi_A[\R(A, \prec)]$ is uniformly complete. Therefore, $C_\le(X_A) = \phi_A[\R(A, \prec)]$, and so $f = \phi_A(c)$ for some $c \in \R(A, \prec)$. Consequently, $a \le c \le b$, and hence $a \prec b$. 

Finally, to see that $\usbal$ is dually equivalent to $\Nach$, by \cref{sbal case} it suffices to show that $\psi : 1_{\usbal} \to \C_\le\Y$ is a natural isomorphism.
Let $S \in \usbal$ and set $(A, \prec) = \B S$. By definition, $\Y S = \X(A, \prec)$, and so $\C_\le\Y S = \C_\le\X(A,\prec)$. We have $S = \R(A, \prec)$, so $A = \B\R(A, \prec)$, and hence $(A, \prec) \in \nbal$. Therefore, by \cref{prop: SW for sbal}, $\phi_A[S]$ is uniformly dense in $C_\le(X_A)$. Because $S \in \usbal$, we get $\phi_A[S] = C_\le(X_A)$. Consequently, $\phi_A|_S$ is an $\sbal$-isomorphism between $S$ and $\C_\le\Y S$. 
Thus, $\usbal$ and $\Nach$ are dually equivalent, and hence $\usbal$ is equivalent to $\unbal$. 
\end{proof}

\begin{remark}
It is a consequence of the above theorem that $\R : \unbal \to \usbal$ is an equivalence, whose quasi-inverse is $\C\circ\Y : \usbal \to \unbal$.
\end{remark}

As we saw in \cref{sec: pbal}, $\B : \sbal\to\spbal$ is an equivalence of categories. We next observe that this equivalence restricts to an equivalence between $\usbal$ and the full subcategory of $\spbal$ consisting of those skeletal $\pbal$-algebras whose skeleton is uniformly complete. 

\begin{definition} \label{def: category pbal}
Let $\uspbal$ be the full subcategory of $\spbal$ consisting of those $(A, \prec)$ whose $\sbal$-skeleton is 
uniformly complete.
\end{definition}

\begin{remark}
While the skeleton of each $(A,\prec) \in \uspbal$ is a uniformly complete $\sbal$-algebra, $A$ itself may not be uniformly complete. To see this, let $X=[0,1]$. Then $C_\le(X) \in \usbal$. On the other hand, $\B C_\le(X)$ is uniformly dense in $C(X)$ by \cref{BC_le dense in C}, and hence $\B C_\le(X)$ is not uniformly complete since $\B C_\le(X) \ne C(X)$ (see, e.g., \cite[p.~117]{RF10}).
\end{remark}

As an immediate consequence of \cref{thm: sbal = bnbal,def: category pbal}, we obtain:

\begin{proposition} \label{prop: usbal = ubnbal}
The equivalence of Theorem~\emph{\ref{thm: sbal = bnbal}} restricts to yield an equivalence between $\usbal$ and $\uspbal$. 
\end{proposition}

\begin{corollary} \plabel{cor: consequences}
\hfill
\begin{enumerate}
\item \label[cor: consequences]{spbal in nbal} $\spbal$ is a full subcategory of $\nbal$.
\item \label[cor: consequences]{spbal = uspbal} The functors $\X : \spbal \to \Nach$ and $\B\C_\le : \Nach \to \spbal$ form a contravariant adjunction, which restricts to a dual equivalence between $\uspbal$ and $\Nach$.
\item \label[cor: consequences]{uspbal = unbal} The functors $\C\X : \uspbal \to \unbal$ and $\B\R : \unbal \to \uspbal$ yield an equivalence between $\uspbal$ and $\unbal$.
\end{enumerate}  
\end{corollary}

\begin{proof}
\ref{spbal in nbal} Let $(A, \prec) \in \spbal$. Then $A = \B\R(A, \prec)$, and hence $\B\R(A,\prec)$ is uniformly dense in $A$. Thus, $(A, \prec) \in \nbal$.

\ref{spbal = uspbal} By \cref{thm: sbal = bnbal}, $\R :\spbal \to \sbal$ and $\B : \sbal \to \spbal$ yield an equivalence between $\spbal$ and $\sbal$. By \cref{sbal case}, there is a contravariant adjunction $\Y : \sbal\to\Nach$ and $\C_\le : \Nach\to\sbal$. Since $\Y \circ \R \cong \X$ (see \cref{prop: functor X}), composing the relevant functors gives a contravariant adjunction $\X : \spbal \to \Nach$ and ${\B \circ \C_\le : \Nach \to \spbal}$, 
which restricts to a dual equivalence between $\uspbal$ and $\Nach$ because the essential image of $\B\circ\C_\le$ is $\uspbal$. 

\ref{uspbal = unbal} By the proof of \cref{thm: main duality}, $\R \cong \C_\le \circ \X$. Therefore, $\B\R \cong \B \circ \C_\le \circ \X$, and the result follows from \ref{spbal = uspbal} and \cref{thm: main duality}.
\end{proof}

Putting together the equivalences and dual equivalences 
obtained in this paper, we arrive at the following diagrams, which commute up to natural isomorphism, 
where 
being a full subcategory is indicated by \hspace{-.25in}$
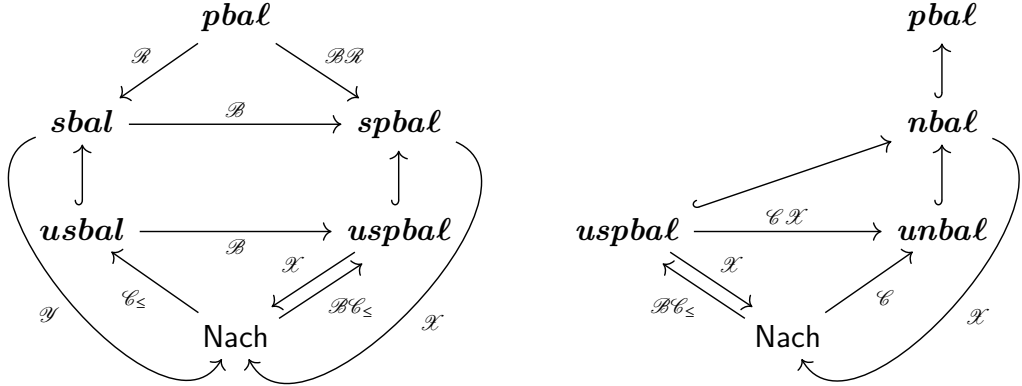
\begin{tikzcd}[column sep = 1.5pc]\phantom{M}\arrow[r, hook]& .\end{tikzcd}$

\[
\begin{tikzcd}[column sep = 1.5pc]
& \pbal \arrow[dl, "\R"'] \arrow[dr, "\B\R"] &&&&& \pbal \\
\sbal \arrow[ddr, bend right = 110, "\Y"'] \arrow[rr, "\B"] && \spbal \arrow[ddl, bend left = 110, "\X"] &&&&\nbal \arrow[u, hookrightarrow] \arrow[ddl, bend left = 110, "\X"] \\
\usbal \arrow[rr, "\B"'] \arrow[u, hookrightarrow] && \uspbal \arrow[u, hookrightarrow] \arrow[dl, shift right = .5ex, "\X"'] &&\uspbal \arrow[rru, hookrightarrow] \arrow[rr, "\C\X"] \arrow[dr, shift left = .5ex, "\X"] && \unbal \arrow[u, hookrightarrow] \\
& \Nach \arrow[ul, "\C_\le"] \arrow[ur, shift right = .5ex, "\B\C_\le"'] &&&& \Nach \arrow[ul, shift left = .5ex, "\B\C_\le"] \arrow[ur, "\C"']
\end{tikzcd}
\]

\section{Comparison with the work of De~Rudder and Hansoul} \label{sec: comparison with DH18}

In this final section we compare our approach to that of De~Rudder and Hansoul \cite{DH18}, and show how to derive their duality from the results of the previous section.
We start by recalling the relevant notions from \cite{DH18}.

\begin{definition} \plabel{def: sbal plus} 
Let $S$ be an $\ell$-monoid satisfying \ref{S4}--\ref{S6}. Suppose there is a unital $\ell$-monoid morphism $\rho : \mathbb{R}^+ \to S$ that preserves multiplication. As in \cref{conv: R inside S}, we identify $\mathbb{R}^+$ with $\rho[\mathbb{R}^+]$ and view $\mathbb{R}^+$ inside $S$ (provided $S$ is nonzero). We call $S$ an {\em $\sbal^+$-algebra} if it is bounded, archimedean, and satisfies 
\begin{enumerate}
\item \label[def: sbal plus]{positive} $0 \le a$ for each $a \in S$;
\item \label[def: sbal plus]{diff} if $a \in S$ and $r \in \mathbb{R}^+$ with $r \le a$, then there is $b \in S$ such that $a = b + r$.
\end{enumerate}
\end{definition}

Similar to the case of $\sbal$, if $S \in \sbal^+$, there is an $\mathbb{R}^+$-action on $S$, given by $ra = \rho(r)a$ for each $a \in S$ and $r \in \mathbb{R}^+$. An $\sbal^+$-morphism is then a unital $\ell$-monoid 
morphism preserving multiplication and the $\mathbb R^+$-action.

\begin{definition} \label{def: usbal plus}
Let $\sbal^+$ be the category of $\sbal^+$-algebras and $\sbal^+$-morphisms. 
\end{definition}

\begin{remark}
Due to \cref{diff},
$\sbal^+$ is a proper subcategory of the category defined in \cite[Def.~2.1]{DH18} (under the name $\sbal$). Our notation is motivated by the fact that each $\sbal^+$-algebra is the positive cone of an $\sbal$-algebra (see \cref{thm: sbal = sbal plus}).  
\end{remark}

\begin{convention}
If $S \in \sbal^+$, we identify $\mathbb{R}^+$ with $\rho[\mathbb{R}^+] \subseteq S$.
\end{convention}

We show that taking the positive cone yields a functor from $\sbal$ to $\sbal^+$.

\begin{proposition} \label{prop: functor from sbal to sbal plus}
There is a functor $\P : \sbal \to \sbal^+$ which sends $S \in \sbal$ to its positive cone $S^+ := \{ a \in S : a \ge 0\}$ and an $\sbal$-morphism $\alpha : S \to T$ to the restriction $\alpha|_{S^+} : S^+ \to T^+$.
\end{proposition}

\begin{proof} 
Let $S \in \sbal$. To see that $S^+ \in \sbal^+$ it is sufficient to verify \cref{diff} since all the other axioms in \cref{def: sbal plus} are obvious. 
Let $a \in S^+$ and $r \in \mathbb{R}^+$ with $r \le a$. Since $0 \le a-r \in S$, 
we have $b := a-r \in S^+$, so $a = b + r$. 
Consequently, $S^+ \in \sbal^+$. It is also clear that if $\alpha : S \to T$ is an $\sbal$-morphism, then $\alpha|_{S^+} : S^+ \to T^+$ is a well-defined $\sbal^+$-morphism, and that $\P$ preserves composition and identity morphisms. Thus, $\P : \sbal \to \sbal^+$ is a functor.
\end{proof}

Let $S \in \sbal^+$. \cref{def: bal envelope} can be employed to define the $\bal$-envelope $\B S$ of $S$, yielding \cite[Def.~2.5]{DH18} (recall that $\B S = \{ [a,b] : a,b \in S^+ \}$). We use $\varepsilon : S \to \B S$ to define the $\sbal$-envelope of $S$. 

\begin{definition} \label{def: S plus R}
For $S \in \sbal^+$, define the {\em $\sbal$-envelope} of $S$ by setting $$\Q S = \{ \varepsilon(a) - r : a \in S, r \in \mathbb{R}^+ \} \subseteq \B S.$$ 
\end{definition}

\begin{lemma} \plabel{lem: properties of S plus R}
Let $S \in \sbal^+$.
\begin{enumerate}
\item \label[lem: properties of S plus R]{S = (S+R) plus} $\varepsilon[S] = (\Q S)^+$. 
\item \label[lem: properties of S plus R]{S + R in sbal} $\Q S \in \sbal$.
\end{enumerate}
\end{lemma}

\begin{proof}
\ref{S = (S+R) plus} The inclusion $\varepsilon[S] \subseteq (\Q S)^+$ is clear. For the reverse inclusion, let $b \in (\Q S)^+$. Then $b = \varepsilon(a) - r$ for some $a \in S$ and $r \in \mathbb{R}^+$. Since $0 \le b$, we have $r \le a$.  By
\cref{diff}, there is $c \in S$ with $a = c + r$. Therefore, since $\varepsilon(r)=r$, we have $b = \varepsilon(a) - r = \varepsilon(c+r) - r = \varepsilon(c) \in \varepsilon[S]$. 

\ref{S + R in sbal} Because $\B S \in \bal$ (see \cref{BS in bal}), it suffices to show that $\Q S$ is an $\sbal$-subalgebra of $\B S$. 
It is clear that $\Q S$ is closed under addition, and it is closed under positive multiplication by \ref{S = (S+R) plus}. For join, let $a, b \in S$ and $r, s \in \mathbb{R}^+$. Then, by \ref{S2}, 
\[
(\varepsilon(a) - r) \vee (\varepsilon(b)-s) = (\varepsilon(a+s)) \vee (\varepsilon(b + r)) - (r + s), 
\]
which belongs to $\Q S$ 
because $S$ is closed under addition and join. A similar calculation shows that $\Q S$ is closed under meet. Finally, let $\rho^+ : \mathbb{R}^+ \to S$ be the $\mathbb{R}^+$-action. We extend 
$\rho^+$ to a map $\rho : \mathbb{R} \to \Q S$ by setting $\rho(-r) = -\rho^+(r)$ for each $r \in \mathbb{R}^+$. 
We show that $\rho$ satisfies \ref{S7}. It is clear that $\rho$ preserves multiplication and 1. For addition, let $r, s \in \mathbb{R}$. There are several cases. Since these cases are proved similarly, we only consider the case when $r \ge 0$, $s \le 0$, and $r + s \ge 0$. In this case, $r = (r+s) + (-s)$, so $\rho^+(r) = \rho^+(r+s) + \rho^+(-s)$. This yields that $\rho(r) = \rho(r+s) - \rho(s)$, and hence $\rho(r+s) = \rho(r) + \rho(s)$. 

It remains to show that $\rho$ preserves join and meet. For this, since $\mathbb R$ is totally ordered, it is enough to show that $\rho$ is order preserving. Let $r, s \in \mathbb{R}$ with $r \le s$. If $r,s$ are both positive, it is clear that 
$\rho(r) \le \rho(s)$. If both $r,s$ are negative, then it is easy to see that $\rho(r) \le \rho(s)$. Finally, if $r$ is negative and $s$ is positive, then $\rho(r) = -\rho^+(-r) \le 0$, so $\rho(r) \le 0 \le \rho(s)$.
\end{proof}

We now define a functor from $\sbal^+$ to $\sbal$.

\begin{proposition}
There is a functor $\Q : \sbal^+ \to \sbal$ which sends $S \in \sbal^+$ to $\Q S$ and an $\sbal^+$-morphism $\alpha : S \to T$ to $\B\alpha|_{\Q S}$.
\end{proposition}

\begin{proof}
If $S \in \sbal^+$, we observed in \cref{S + R in sbal} that $\Q S \in \sbal$. Let $\alpha : S \to T$ be an $\sbal^+$-morphism. Then $\B\alpha : \B S \to \B T$ is a $\bal$-morphism (see \cref{rem: B is a functor}). Let $a \in S$ and $r \in \mathbb{R}^+$. Then $\B\alpha(s-r) = \alpha(s) - r \in \Q T$. Therefore, $\B\alpha|_{\Q S} : \Q S \to \Q T$ is well defined, and it is straightforward to see that it is an $\sbal$-morphism. It is clear that $\Q$ preserves composition and identity morphisms. Thus, $\Q$ is a functor.
\end{proof}

\begin{definition}
Let $\usbal^+$ be the full subcategory of $\sbal^+$ consisting of those $S \in \sbal^+$ whose image $\varepsilon[S]$ is complete with respect to the uniform norm on $\B S$. 
\end{definition}

\begin{remark}
The category $\usbal^+$ is the category defined in  \cite[Def.~4.6]{DH18} (under the name $\usbal$). 
\end{remark}

\begin{theorem} \label{thm: sbal = sbal plus}
The functors $\P$ and $\Q$ yield an equivalence 
between $\sbal$ and $\sbal^+$, and restrict to an equivalence between $\usbal$ and $\usbal^+$. 
\end{theorem}

\begin{proof}
Let $S \in \sbal$. We show that $\varepsilon[S] = \Q (S^+)$. First let $a \in S$. Since $S$ is bounded, there is $r \in \mathbb{R}^+$ with $a + r \ge 0$. Therefore, $\varepsilon(a) = \varepsilon(a+r) - r \in \Q(S^+)$, yielding that $\varepsilon[S] \subseteq \Q (S^+)$. For the reverse inclusion, let $a \in S^+$ and $r \in \mathbb{R}^+$. Then $\varepsilon(a) - r = \varepsilon(a-r) \in \varepsilon[S]$. Thus, $\varepsilon[S] = \Q(S^+) = \Q\P(S)$. Consequently, there is an $\sbal$-isomorphism $\lambda_S =\varepsilon : S \to \Q\P(S)$.

Next, let $S \in \sbal^+$. Then $\varepsilon[S] = (\Q S)^+$ by \cref{S = (S+R) plus}. Therefore, there is an $\sbal^+$-isomorphism $\mu_S = \varepsilon : S \to \P\Q(S)$. 
Naturality of $\lambda$ and $\mu$ is straightforward. Thus, $\sbal$ and $\sbal^+$ are equivalent. 

It is left to show that the functors $\P, \Q$ restrict to an equivalence of $\usbal$ and $\usbal^+$. It suffices to show that for $S \in \sbal^+$, $\Q S \in \usbal$ iff $S \in \usbal^+$.  
Suppose that $\Q S \in \usbal$. 
Let $\{a_n\}$ be a Cauchy sequence in $\varepsilon[S]$. Then it is a Cauchy sequence in $\Q S$, so it converges in $\Q S$. Because each element of $\varepsilon[S]$ is nonnegative, the limit is nonnegative, so it lies in $(\Q S)^+$, and hence in $\varepsilon[S]$ by \cref{S = (S+R) plus}. Therefore, $\{a_n\}$ converges in $\varepsilon[S]$, and hence $S \in \usbal^+$.

Conversely, suppose that $S \in \usbal^+$. Let $\{\varepsilon(a_n) - r_n\}$ be a Cauchy sequence in $\Q S$. It is then a bounded sequence, so there is $s \in \mathbb{R}^+$ with $a_n - r_n + s \ge 0$ for each $n$. By \cref{S = (S+R) plus}, there are $b_n \in S$ with $b_n = a_n - r_n + s$. The sequence $\{b_n\}$ is Cauchy in $S$, so it converges to some $b \in S$ because $S \in \usbal^+$. Therefore, $\{a_n - r_n + s\}$ converges to $b$, and so $\{\varepsilon(a_n) - r_n\}$ converges to $\varepsilon(b) - s \in \Q S$. Thus, $\Q S \in \usbal$.
\end{proof}

As an immediate consequence of \cref{thm: main duality,thm: sbal = sbal plus}, we obtain:

\begin{corollary} \cite[Thm.~4.9]{DH18}
The categories $\usbal^+$ and $\Nach$ are dually equivalent.
\end{corollary}

\begin{remark}
The functors establishing the dual equivalence of the previous corollary are the compositions 
$\Y\circ\Q : \usbal^+ \to \Nach$ and $\P \circ \C_\le : \Nach \to \usbal^+$.
\end{remark}

\newcommand{\etalchar}[1]{$^{#1}$}
\def\cprime{$'$}

\end{document}